\documentclass[11pt, reqno]{amsart}

\usepackage{amssymb}
\usepackage{eucal}
\usepackage{amsmath}
\usepackage{amscd}
\usepackage[dvips]{color}
\usepackage{multicol}
\usepackage[all]{xy}           
\usepackage{graphicx}
\usepackage{color}
\usepackage{colordvi}
\usepackage{xspace}
\usepackage{tikz}

\def\a{\alpha}

\def\ci{\circ}

\def\d{\delta}

\def\di{\diamond}
\def\D{\Delta}

\def\g{\gamma}
\def\l{\lambda}
\def\lr{\longrightarrow}
\def\o{\otimes}

\def\rc{\prec}
\def\lc{\succ}

\def\r{\rho}

\def\s{\sigma}

\def\tl{\triangleright}
\def\tr{\triangleleft}
\def\v{\varepsilon}

\def\1{^{-1}}
\def\2{^{-2}}
\def\3{^{-3}}

\usepackage[active]{srcltx} 

\renewcommand\baselinestretch{1}    

\begin{document}
\renewcommand{\baselinestretch}{1.2}
 \renewcommand{\arraystretch}{1.0}

\newcommand {\emptycomment}[1]{} 

\newcommand{\nc}{\newcommand}
\newcommand{\delete}[1]{}
\nc{\mfootnote}[1]{\footnote{#1}} 
\nc{\todo}[1]{\tred{To do:} #1}

\nc{\mlabel}[1]{\label{#1}}  
\nc{\mcite}[1]{\cite{#1}}  
\nc{\mref}[1]{\ref{#1}}  
\nc{\mbibitem}[1]{\bibitem{#1}} 

\delete{
\nc{\mlabel}[1]{\label{#1}  
{\hfill \hspace{1cm}{\bf{{\ }\hfill(#1)}}}}
\nc{\mcite}[1]{\cite{#1}{{\bf{{\ }(#1)}}}}  
\nc{\mref}[1]{\ref{#1}{{\bf{{\ }(#1)}}}}  
\nc{\mbibitem}[1]{\bibitem[\bf #1]{#1}} 
}

\newtheorem{thm}{Theorem}[section]
\newtheorem{lem}[thm]{Lemma}
\newtheorem{cor}[thm]{Corollary}
\newtheorem{pro}[thm]{Proposition}
\newtheorem{ex}[thm]{Example}
\newtheorem{rmk}[thm]{Remark}
\newtheorem{defi}[thm]{Definition}
\newtheorem{pdef}[thm]{Proposition-Definition}
\newtheorem{condition}[thm]{Condition}
\newtheorem{con}[thm]{Conclusion}

\renewcommand{\labelenumi}{{\rm(\alph{enumi})}}
\renewcommand{\theenumi}{\alph{enumi}}

\nc{\tred}[1]{\textcolor{red}{#1}}
\nc{\tblue}[1]{\textcolor{blue}{#1}}
\nc{\tgreen}[1]{\textcolor{green}{#1}}
\nc{\tpurple}[1]{\textcolor{purple}{#1}}
\nc{\btred}[1]{\textcolor{red}{\bf #1}}
\nc{\btblue}[1]{\textcolor{blue}{\bf #1}}
\nc{\btgreen}[1]{\textcolor{green}{\bf #1}}
\nc{\btpurple}[1]{\textcolor{purple}{\bf #1}}

\nc{\dcy}[1]{\textcolor{purple}{Chengyu:#1}}
\nc{\cm}[1]{\textcolor{red}{Chengming:#1}}
\nc{\li}[1]{\textcolor{blue}{Li: #1}}
\nc{\lit}[2]{\textcolor{blue}{#1}{}} 
\nc{\yh}[1]{\textcolor{green}{Yunhe: #1}}

\title{Rota-Baxter bisystems and covariant bialgebras}

\author{Tianshui Ma}
\address{School of Mathematics and Information Science, Henan Normal University, Xinxiang 453007, China}
\email{matianshui@yahoo.com}

\author{Abdenacer Makhlouf}
\address{Universit{\'e} de Haute Alsace, Laboratoire de Math{\'e}matiques, Informatique et Applications, 4, rue des Fr{\`e}res Lumi{\`e}re F-68093 Mulhouse, France}
\email{abdenacer.makhlouf@uha.fr}

\author{Sergei Silvestrov} \thanks{{\it Corresponding author}: Sergei Silvestrov, sergei.silvestrov@mdh.se}
\address{Division of Applied Mathematics, School of Education,
Culture and Communication, M\"alardalen University, 72123 V\"aster\"as, Sweden}
\email{sergei.silvestrov@mdh.se}

\date{August 17, 2017}

\begin{abstract}
 The aim of this paper is first to introduce and study Rota-Baxter cosystems and bisystems  as generalization of Rota-Baxter coalgebras and bialgebras, respectively, with various examples. The second purpose is to provide an alternative definition of covariant bialgebras via coderivations. Furthermore, we consider  coquasitriangular covariant bialgebras which  are generalization of coquasitriangular infinitesimal bialgebras, coassociative Yang-Baxter pairs, coassociative Yang-Baxter equation, double coalgebras and dendriform coalgebras.  We study some of  their properties and relationships with Rota-Baxter cosystems and bisystems.
\end{abstract}

\subjclass{16T05, 16W99}

\keywords{Rota-Baxter cosystem; Rota-Baxter bisystem; Covariant bialgebra; Yang-Baxter equation; Coquasitriangular covariant bialgebra; Coquasitriangular infinitesimal bialgebra.}

\maketitle

\tableofcontents

\numberwithin{equation}{section}

\tableofcontents
\numberwithin{equation}{section}
\allowdisplaybreaks

\section{Introduction} \label{sec:Intro}
 Rota-Baxter algebras were introduced in \cite{Ro} in the context of differential operators on commutative Banach algebras and then led to an algebraic interpretation of Spitzer identity in fluctuation theory in probability and combinatorics \cite{Ba}. Since then,  they were  intensively studied and appeared in various areas such as mathematical physics, mainly in Connes-Kreimer renormalization theory in quantum field theory   (see \cite{EFG1,EFG2,EFG3,Gu00,GK00}), Lie algebras (see \cite{AB}), multiple zeta functions (see \cite{EFG1,GZ}), differential algebras (see \cite{GK08}). One can refer to the book \cite{Gu12} for the detailed theory of Rota-Baxter algebras. In 2014, based on the dual method in the Hopf algebra theory, Jian and Zhang defined  in \cite{JZ} the notion of Rota-Baxter coalgebras and also provided various examples. Therefore Rota-Baxter bialgebras were presented in \cite{ML} with examples constructed using Radford biproduct. In \cite{Br}, T. Brzezi\'{n}ski introduced the notion of Rota-Baxter system, where Rota-Baxter algebras can be viewed as  special case. Following the above idea, we aim in this paper  to study the dual version and introduce Rota-Baxter co-(bi-)systems.

 An associative (not necessarily unital) algebra $A$ that admits a coassociative comultiplication which is a derivation is called infinitesimal bialgebra, see \cite{Ag}. It is self-dual and  can be viewed as a coalgebra that admits an associative multiplication which is a coderivation. In \cite{Br}, T. Brzezi\'{n}ski proposed the concept of covariant bialgebra via two derivations. It is  an extension of infinitesimal bialgebras and quasitriangular covariant bialgebras obtained by Yang-Baxter pairs. It turns out that  covariant bialgebras are not self-dual. We provide  another approach to study covariant bialgebras through two coderivations, and at the same time study  similarly coquasitriangular infinitesimal bialgebras.

 This paper is organized as follows. In Section \ref{sec:Prelim}, we review some preliminaries and in Section \ref{sec:RBcosysbisysbialg}, we discuss the notions of Rota-Baxter cosystem, bisystem and relate them to Rota-Baxter coalgebras, pre-Lie coalgebras, weak copseudotwistors and Rota-Baxter bialgebras. Section \ref{sec:covarbialg} is devoted  to  study  covariant bialgebras via coderivations. Firstly, we introduce the notion of coassociative Yang-Baxter pair (abbr. CYBP), and then investigate their  relationships with  Rota-Baxter cosystems and bisystems. Secondly, we give an alternative definition of  covariant bialgebras through two coderivations and provide a  characterization. In Section \ref{sec:CoquasitriangCovBialgInfinBialg}, as a special case, coquasitriangular covariant bialgebras are considered. They are a generalization  coquasitriangular infinitesimal bialgebras. On the other hand, we introduce  the concept of right (left) covariant modules.  Furthermore, we consider  coassociative Yang-Baxter equation, double coalgebras and dendriform coalgebras.  We study some of  their properties and relationships with cosystems and bisystems. In Section \ref{sec:ExamplesRBbialgbisyst}, we provide various examples of Rota-Baxter bialgebras and bisystems in dimensions 2, 3 and 4.

\section{Preliminaries} \label{sec:Prelim}
 Throughout this paper, we assume that all vector spaces, algebras, coalgebras and tensor products are defined over a field $K$. An algebra is always assumed to be associative, but not necessarily unital. A coalgebra is always assumed to be coassociative, but not necessarily counital. Now, for a given a coalgebra $C$, we use Sweedler's notation for the comultiplication : $\D(c)=c_{1}\o c_{2}$ for any $c\in C$, see \cite{Ra12}. We denote the category of left $H$-comodules (resp. right $H$-comodules) by $^H \hbox{Mod}$ (resp.  $ \hbox{Mod}^H$). For $(M, \rho_l)\in$ $^H \hbox{Mod}$, we write: $\rho_l(x)=x_{(-1)}\o x_{(0)} \in H\o M$, for all $x \in M$. For $(M, \rho_r)\in$ $ \hbox{Mod}^H$, we write: $ \rho_r(x)=x_{(0)}\o x_{(1)} \in M\o H$, for all $x \in M$.  Given a $K$-space $M$, we write $id_M$ for the identity map on $M$. We say that a coalgebra $C$ is {\bf non-degenerate} if provided with a  linear map $f: C\lr C$ such that  for $\forall~c\in C$,  $c_1\o f(c_2)=0$ or $f(c_1)\o c_2=0$,  it  implies that $f=0$. Obviously, any counital coalgebra is non-degenerate.

 In what follows, we recall  some useful definitions which will be used later (see \cite{Br,Gu12,ML}).

\begin{defi}\label{de:2.1}
 {\rm  For a given field $K$ and $\l\in K$, a {\bf Rota-Baxter algebra of weight $\l$} is an algebra $A$ together with a $K$-linear map $R : A \lr A$ such that
 \begin{equation}\label{eq:2.1}
 R(a)R(b)=R(a R(b))+R(R(a) b)+\l R(a b)
 \end{equation}
 for all $a, b\in A$. Such a linear operator is called a {\bf Rota-Baxter operator of weight $\l$ on $A$}.}
\end{defi}

\begin{defi}\label{de:2.2}
 {\rm  Let $\g$ be an element in $K$. A pair $(C, Q)$ is called a {\bf Rota-Baxter coalgebra of weight $\gamma$} if $C$ is a coalgebra and $Q$ is a linear endomorphism of $C$ satisfying
 $$
 (Q\o Q)\D=(id\o Q)\D Q+(Q\o id)\D Q+\gamma\D Q,
 $$
 that is, for all $c\in C$,
 \begin{equation}\label{eq:2.2}
 Q(c_1)\o Q(c_2)=Q(c)_1\o Q(Q(c)_2)+Q(Q(c)_1)\o Q(c)_2+ \gamma Q(c)_1\o Q(c)_2.
 \end{equation}
 The map $Q$ is called a {\bf Rota-Baxter operator of weight $\gamma$ on $C$}. }
\end{defi}


\begin{defi}\label{de:2.3}
 {\rm  Let $\lambda$, $\gamma$ be elements in $K$ and $H$ a bialgebra (not necessarily  unital and counital). A triple $(H, R, Q)$ is called a {\bf Rota-Baxter bialgebra of weight $(\lambda, \gamma)$} if $(H, R)$ is a Rota-Baxter algebra of weight $\lambda$ and $(H, Q)$ is a Rota-Baxter coalgebra of weight $\gamma$. We say that we have a $(R,Q)$-Rota-Baxter structure of weight $(\lambda, \gamma)$ on the  bialgebra $H$.}
\end{defi}

\begin{defi}\label{de:2.4}
 {\rm  A triple $(A, R, S)$ consisting of an algebra $A$ and two $K$-linear operators $R, S: A \lr A$ is called a {\bf Rota-Baxter system} if, for all $a, b \in A$,
 \begin{eqnarray}
 R(a)R(b)&=&R(R(a)b+aS(b)), \label{eq:2.3}\\
 S(a)S(b)&=&S(R(a)b+aS(b)). \label{eq:2.4}
 \end{eqnarray}}
\end{defi}

\begin{defi}\label{de:2.5}
 {\rm  A {\bf pre-Lie coalgebra} is a pair  $(C,\D)$ consisting of a linear space $C$ and a  linear map $\D: C\lr C\o C$ satisfying
 \begin{equation}\label{eq:2.5}
 \D_{C}-\Phi_{(12)}\D_{C}=0,
 \end{equation}
 where $\D_{C}=(\D\o id)\D-(id\o\D)\D$ and $\Phi_{(12)}(c_1\o c_2\o c_3)=c_2\o c_1\o c_3$.}
\end{defi}

\section{Rota-Baxter cosystems, Rota-Baxter bisystems and Rota-Baxter bialgebras} \label{sec:RBcosysbisysbialg}
 In this section we introduce  Rota-Baxter cosystems, bisystems and relate them to Rota-Baxter coalgebras, bialgebras and pre-Lie coalgebras.

\begin{defi}\label{de:3.1}
 {\rm  A triple $(C, Q, T)$ consisting of a coalgebra $C$ and two $K$-linear operators $Q, T: C\lr C$ is called a {\bf Rota-Baxter cosystem} if, for all $c\in C$,
 \begin{eqnarray}
 Q(c_1)\o Q(c_2)&=&Q(Q(c)_1)\o Q(c)_2+Q(c)_1\o T(Q(c)_2),  \label{eq:3.1}\\
 T(c_1)\o T(c_2)&=&Q(T(c)_1)\o T(c)_2+T(c)_1\o T(T(c)_2).  \label{eq:3.2}
 \end{eqnarray}

 A {\bf morphism of Rota-Baxter cosystems between $(C, Q_C, T_C)$ and $(D, Q_D, T_D)$} is a coalgebra map $f: C\lr D$ such that $f\ci Q_C=Q_D\ci f$ and $f\ci T_C=T_D\ci f$.}
\end{defi}

\begin{ex}\label{ex:3.2}
 Let $C$ be a coassociative coalgebra, $g: C\lr C$ a comultiplicative map and $Q: C\lr C$ a $K$-linear map such that for all $c\in C$,
 \begin{equation}\label{eq:3.3}
 Q(c_1)\o Q(c_2)=Q(Q(c)_1)\o Q(c_2)+Q(c)_1\o Q^g(Q(c)_2),
 \end{equation}
 where $Q^g=Q\ci g$. Then $(C, Q, Q^g)$ is a Rota-Baxter cosystem.

 \begin{proof} Eq.(\ref{eq:3.3}) is exactly Eq.(\ref{eq:3.1}) for $(Q, Q^g)$ and Eq.(\ref{eq:3.2}) can be checked by Eq.(\ref{eq:3.3}) and $g: C\lr C$ is a comultiplicative map. \end{proof}
\end{ex}

 The following proposition states that Rota-Baxter coalgebra can be viewed as a special case of Rota-Baxter cosystem.

\begin{pro}\label{pro:3.3}
 Let $C$ be a coalgebra. If $(C, Q)$ is a Rota-Baxter coalgebra of weight $\g$, then $(C, Q, Q+\g id)$ and $(C, Q+\g id, Q)$ are Rota-Baxter cosystems.
\end{pro}

\begin{proof} We only check that $(C, Q, Q+\g id)$ is a Rota-Baxter cosystem as follows:
 \begin{eqnarray*}
 \hbox{RHS of Eq.(\ref{eq:3.2})}&\stackrel{T=Q+\g id}{=}&Q(Q(c)_1)\o Q(c)_2+\g Q(c)_1\o c_2+Q(c)_1\o Q(Q(c)_2)\\
 &&+\g Q(c)_1\o Q(c)_2+\g c_1\o Q(c_2)+\g^2 c_1\o c_2\\
 &\stackrel{(\ref{eq:2.2})}{=}&Q(c_1)\o Q(c_2)+\g Q(c)_1\o c_2+\g c_1\o Q(c_2)+\g^2 c_1\o c_2\\
 &\stackrel{}{=}&(Q(c_1)+\g c_1)\o (Q(c_2)+\g c_2)\\
 &\stackrel{}{=}&\hbox{LHS of Eq.(\ref{eq:3.2})}.
 \end{eqnarray*}
 It is obvious that Eq.(\ref{eq:3.1}) holds for $(C, Q, Q+\g id)$ by using Eq.(\ref{eq:2.2}). The proof for $(C, Q+\g id, Q)$ is similar.    \end{proof}

\begin{pro}\label{pro:3.4}
 Let $C$ be a coalgebra, $Q: C\lr C$ a left $C$-colinear map and let $T: C\lr C$ be a right $C$-colinear map. Then $(C, Q, T)$ is a Rota-Baxter cosystem if and only if, for all $c\in C$,
 \begin{equation}\label{eq:3.4}
 c_1\o (T\ci Q)(c_2)=(Q\ci T)(c_1)\o c_2=0.
 \end{equation}
 In particular, if $C$ is a non-degenerate coalgebra, then $(C, Q, T)$ is a Rota-Baxter cosystem (with $Q$ left and $T$ right $C$-colinear) if and only if $Q$ and $T$ satisfy the orthogonality condition
 \begin{equation}\label{eq:3.5}
 T\ci Q=Q\ci T=0.
 \end{equation}
\end{pro}

\begin{proof} $Q$ is left $C$-colinear and $T$ is right $C$-colinear, that is,
 $Q(c)_1\o Q(c)_2=c_1\o Q(c_2)$ and $T(c)_1\o T(c)_2=T(c_1)\o  c_2$. Thus Eq.(\ref{eq:3.1}) holds if and only if $c_1\o (T\ci Q)(c_2)=0$ and Eq.(\ref{eq:3.2}) holds if and only if $(Q\ci T)(c_1)\o c_2=0$. Since $C$ is non-degenerate $Q\ci T=0$.
  \end{proof}

\begin{ex}\label{ex:3.5}
 Let $C$ be a coalgebra. Assume that $\s, \tau\in C^*$ are such that for $\forall~c\in C$, $\s(c_1)\tau(c_2)=0$ and $\s(c_1)c_2=c_1\s(c_2)$. Define $Q, T: C\lr C$ by
 $$
 Q(c)=c_1\s(c_2),~~~~T(c)=\tau(c_1)c_2
 $$
 for all $c\in C$. Then it is obvious that $Q$ is left $C$-colinear and $T$ is right $C$-colinear since the comodule action is the coproduct in $C$. While
 $$
 c_1\o (T\ci Q)(c_2)=c_1\o \tau(c_2)c_3\s(c_4)=c_1\o \s(c_2)\tau(c_3)c_4=0
 $$
 and
 $$
 (Q\ci T)(c_1)\o c_2=\tau(c_1)c_2\s(c_3)\o c_4=\s(c_1)\tau(c_2)c_3\o c_4=0.
 $$
 Therefore $(C, Q, T)$ is a Rota-Baxter cosystem.
\end{ex}

\begin{ex} \label{ex:3.6}
We construct Rota-Baxter bisystems for $T_2$, the unital Taft-Sweedler algebra generated by $g,x$ and the relations $(g^2=1,\ x^2=0, \ x g=-g x).$ The comultiplication is defined by $\Delta (g)=g\otimes g$ and   $\Delta (x)=x\otimes 1+g\otimes x$, the counit is given by $\varepsilon(g)=1,\ \varepsilon (x)=0.$

The bialgebra $T_2$ is 4-dimensional, it is defined with respect to a basis
$
\{u_1=1, \ u_2=g,\ u_3=x,\ u_4=gx\}
$
by the following table
which  describes multiplying the $i$th row elements  by the $j$th column elements.

\[
\begin{array}{|c|c|c|c|c|}
  \hline
   \ & u_1& u_2 & u_3 & u_4  \\ \hline
   u_1& u_1& u_2 & u_3 & u_4 \\ \hline
   u_2 &u_2 & u_1 & u_4 & u_3 \\ \hline
   u_3 &u_3 & -u_4 & 0 & 0  \\ \hline
   u_4 &u_4 & -u_3 & 0 & 0  \\ \hline
\end{array}
\]
and
\begin{eqnarray*}
&& \Delta(u_1)=u_1 \otimes u_1,\  \Delta(u_2)=u_2 \otimes u_2,\ \Delta(u_3)=u_1 \otimes u_3+u_3 \otimes u_2,\ \Delta(u_4)=u_2 \otimes u_4+u_4 \otimes u_1.
\end{eqnarray*}
\begin{equation*}
\varepsilon (u_1)=\varepsilon (u_2)=1,\quad \varepsilon (u_3)=\varepsilon (u_4)=0.
\end{equation*}

A bisystem is given by a pair $(R,S)$ and a pair $(Q,T)$ taken from the following list for $(R,S)$

\begin{itemize}
\item $R(u_1)=0, \  R(u_2)=0, \ R(u_3)=0, \ R(u_4)=p u_4, $
\\  $S(u_1)=-p u_1, \  S(u_2)=0, \ S(u_3)=0, \ S(u_4)=0, $
\item $R(u_1)=0, \  R(u_2)=0, \ R(u_3)=-p u_3, \ R(u_4)=-p u_4, $
\\  $S(u_1)=p u_1, \  S(u_2)=p u_2, \ S(u_3)=0, \ S(u_4)=0, $
\item $R(u_1)=0, \  R(u_2)=0, \ R(u_3)=p u_3, \ R(u_4)=0, $
\\  $S(u_1)=-p u_1, \  S(u_2)=0, \ S(u_3)=0, \ S(u_4)=0, $
\item $R(u_1)=-p u_1, \  R(u_2)=0, \ R(u_3)=0, \ R(u_4)=0, $
\\  $S(u_1)=0, \  S(u_2)=0, \ S(u_3)=p u_3, \ S(u_4)=0, $
\item $R(u_1)=-p u_1, \  R(u_2)=-p u_2, \ R(u_3)=0, \ R(u_4)=0, $
\\  $S(u_1)=0, \  S(u_2)=0, \ S(u_3)=p u_3, \ S(u_4)=p u_3, $
\item $R(u_1)=-p u_1, \  R(u_2)=-p u_2, \ R(u_3)=0, \ R(u_4)=0, $
\\  $S(u_1)=0, \  S(u_2)=0, \ S(u_3)=p u_3, \ S(u_4)=-p u_3, $
\item $R(u_1)=-p u_1, \  R(u_2)=0, \ R(u_3)=0, \ R(u_4)=0, $
\\  $S(u_1)=0, \  S(u_2)=0, \ S(u_3)=0, \ S(u_4)=p u_4, $
\item $R(u_1)=p u_1, \  R(u_2)=0, \ R(u_3)=0, \ R(u_4)=0, $
\\  $S(u_1)=0, \  S(u_2)=0, \ S(u_3)=0, \ S(u_4)=0, $
\end{itemize}
where   $p$ is a parameter, and the following list for $(Q,T)$

\begin{itemize}
\item $Q(u_1)=0, \  Q(u_2)=0, \ Q(u_3)=0, \ Q(u_4)=-q_1 u_4, $
\\  $T(u_1)=q_1 u_1, \  T(u_2)=q_2 u_2, \ T(u_3)=q_2 u_3, \ T(u_4)=0, $
\item $Q(u_1)=0, \  Q(u_2)=0, \ Q(u_3)=0, \ Q(u_4)=-q_1 u_4, $
\\  $T(u_1)=q_1 u_1, \  T(u_2)=q_2 u_2, \ T(u_3)=0, \ T(u_4)=0, $
\item $Q(u_1)=0, \  Q(u_2)=0, \ Q(u_3)=-q_1 u_3, \ Q(u_4)=0, $
\\  $T(u_1)=q_1 u_1, \  T(u_2)=q_2 u_2, \ T(u_3)=0, \ T(u_4)=q_2 u_4, $
\item $Q(u_1)=0, \  Q(u_2)=0, \ Q(u_3)=-q_1 u_3, \ Q(u_4)=0, $
\\  $T(u_1)=q_1 u_1, \  T(u_2)=q_2 u_2, \ T(u_3)=0, \ T(u_4)=0, $
\item $Q(u_1)=0, \  Q(u_2)=0, \ Q(u_3)=-q_1 u_3, \ Q(u_4)=-q_1 u_4, $
\\  $T(u_1)=q_1 u_1, \  T(u_2)=q_2 u_2, \ T(u_3)=0, \ T(u_4)=0, $
\item $Q(u_1)=0, \  Q(u_2)=-q_1 u_2, \ Q(u_3)=0, \ Q(u_4)=0, $
\\  $T(u_1)=q_1 u_1, \  T(u_2)=0, \ T(u_3)=q_2 u_3, \ T(u_4)=0, $
\item $Q(u_1)=0, \  Q(u_2)=-q_1 u_2, \ Q(u_3)=0, \ Q(u_4)=0, $
\\  $T(u_1)=q_2 u_1, \  T(u_2)=0, \ T(u_3)=q_1 u_3, \ T(u_4)=q_1 u_4, $
\item $Q(u_1)=0, \  Q(u_2)=-q_1 u_2, \ Q(u_3)=0, \ Q(u_4)=-q_2 u_4, $
\\  $T(u_1)=q_2 u_1, \  T(u_2)=0, \ T(u_3)=q_1 u_1, \ T(u_4)=0, $
\end{itemize}
where   $q_1,q_2$ are parameters.
\end{ex}

\begin{pro}\label{pro:3.6}
 Let $(C, Q, T)$ be a Rota-Baxter cosystem. Then

 (1) $(C, \D_{*})$ with $\D_{*}: C\lr C\o C$, defined by
 \begin{equation}\label{eq:3.6}
 \D_{*}(c)=Q(c_1)\o c_2+c_1\o T(c_2)
 \end{equation}
 for all $c\in C$, is a coassociative coalgebra.

 (2) $(C, \D_{\bullet})$ with $\D_{\bullet}: C\lr C\o C$, defined by
 \begin{equation}\label{eq:3.7}
 \D_{\bullet}(c)=Q(c_1)\o c_2-T(c_2)\o c_1
 \end{equation}
 for all $c\in C$, is a pre-Lie coalgebra.
\end{pro}

\begin{proof}
 (1) The coassociativity for $(C, \D_{*})$ can be proved as follows. For all $c\in C$, we have
 \begin{eqnarray*}
 (\D_{*}\o id)\ci \D_{*}(c)&\stackrel{}{=}&Q(Q(c_1)_1)\o Q(c_1)_2\o c_2+Q(c_1)_1\o T(Q(c_1)_2)\o c_2\\
 &&+Q(c_{11})\o c_{12}\o T(c_2)+c_{11}\o T(c_{12}\o T(c_2)\\
 &\stackrel{(\ref{eq:3.2})}{=}&Q(Q(c_1)_1)\o Q(c_1)_2\o c_2+Q(c_1)_1\o T(Q(c_1)_2)\o c_2\\
 &&+Q(c_{11})\o c_{12}\o T(c_2)+c_{1}\o Q(T(c_{2})_1)\o T(c_2)_2\\
 &&+c_{1}\o T(c_{2})_1\o T(T(c_2)_2)\\
 &\stackrel{(\ref{eq:3.1})}{=}&Q(c_1)\o Q(c_{21})\o c_{22}+Q(c_1)\o c_{21}\o T(c_{22})\\
 &&+c_{1}\o Q(T(c_{2})_1)\o T(c_2)_2+c_{1}\o T(c_{2})_1\o T(T(c_2)_2)\\
 &\stackrel{}{=}&(id\o \D_{*})\ci \D_{*}(c).
 \end{eqnarray*}

 (2) Since
 \begin{eqnarray*}
 \D_{C}(c)&\stackrel{}{=}&((\D_{\bullet}\o id)\ci \D_{\bullet}-(id\o \D_{\bullet})\ci \D_{\bullet})(c)\\
 &\stackrel{}{=}&Q(Q(c_1)_1)\o Q(c_1)_2\o c_2-T(Q(c_1)_2)\o Q(c_1)_1\o c_2\\
 &&-Q(T(c_{2})_1)\o T(c_2)_2\o c_{1}+T(T(c_2)_2)\o T(c_{2})_1\o c_{1}\\
 &&-Q(c_{1})\o Q(c_{21})\o c_{22}+Q(c_1)\o T(c_{22})\o c_{21}\\
 &&+T(c_{2})\o Q(c_{11})\o c_{12}-T(c_{2})\o T(c_{12})\o c_{11}\\
 &\stackrel{(\ref{eq:3.1})(\ref{eq:3.2})}{=}&-T(Q(c_1)_2)\o Q(c_1)_1\o c_2-Q(T(c_{2})_1)\o T(c_2)_2\o c_{1}\\
 &&-Q(c_1)_1\o T(Q(c_1)_2)\o c_2-Q(c_{1})\o T(c_{22})\o c_{21}\\
 &&+T(c_{2})\o Q(c_{11})\o c_{12}-T(c_{2})_{2}\o Q(T(c_{2})_{1})\o c_{1},
 \end{eqnarray*}
 we can get $(\D_{C}-\Phi_{(12)}\D_{C})(c)=0$. Therefore $(C, \D_{\bullet})$ is a pre-Lie coalgebra. \end{proof}

 \begin{rmk}\label{rmk:3.7}
 Here the pre-Lie coalgebra is dual to pre-Lie algebra in \cite{AB}.
 \end{rmk}

 \smallskip

 Next we introduce the notion of weak copseudotwistor and relate it to Rota-Baxter cosystem.

\begin{defi}\label{de:3.8}
 {\rm  Let $C$ be a coalgebra with coassociative coproduct $\D: C\lr C\o C$. A $K$-linear map $F: C\o C\lr C\o C$ is called a {\bf weak copseudotwistor} if there exists a $K$-linear map $\widetilde{F}: C\o C\o C\lr C\o C\o C$, rendering commutative the following diagram:

\begin{equation}\label{eq:3.8}
\unitlength 1mm 
\linethickness{0.4pt}
\ifx\plotpoint\undefined\newsavebox{\plotpoint}\fi 
\begin{picture}(129,39.75)(0,0)
\put(53,6.75){\line(-1,0){.25}}
\put(53,37.5){\line(-1,0){.25}}
\put(77.25,7){\vector(1,0){14}}
\put(77.25,37.75){\vector(1,0){14}}
\put(61,7.5){\vector(-1,0){14}}
\put(61,38.25){\vector(-1,0){14}}
\put(64.25,6.5){$C\o C$}
\put(64.25,37.25){$C\o C$}
\put(27.75,6.75){$C\o C\o C$}
\put(27.75,37.5){$C\o C\o C$}
\put(1.75,21.25){$C\o C\o C$}
\put(95.25,6.75){$C\o C\o C$}
\put(95.25,37.5){$C\o C\o C$}
\put(124.5,20.75){$C\o C\o C$.}
\put(128,25.75){\vector(3,-1){.07}}\multiput(103.75,34.75)(.09082397,-.033707865){267}{\line(1,0){.09082397}}
\put(129,19){\vector(4,1){.07}}\multiput(103.75,11.75)(.11744186,.03372093){215}{\line(1,0){.11744186}}
\put(13.25,25.75){\vector(-3,-1){.07}}\multiput(37.25,35)(-.087272727,-.033636364){275}{\line(-1,0){.087272727}}
\put(13,19.25){\vector(-3,1){.07}}\multiput(38,10.25)(-.093632959,.033707865){267}{\line(-1,0){.093632959}}
\put(68.25,10.75){\vector(0,1){23.25}}
\put(70.25,21){$F$}
\put(118,10.75){$\widetilde{F}$}
\put(13.75,11.75){$\widetilde{F}$}
\put(49.25,9.25){$id\o \D$}
\put(49.25,39.25){$id\o \D$}
\put(77.75,9.25){$\D\o id$}
\put(77.75,39.75){$\D\o id$}
\put(13.25,31.5){$id\o F$}
\put(117.5,32){$F\o id$}
\end{picture}
\end{equation}
 The map $\widetilde{F}$ is called a {\bf weak companion of $F$}.}
 \end{defi}

\begin{pro}\label{pro:3.9}  If $(C, Q, T)$ is a Rota-Baxter cosystem, then
 \begin{equation}\label{eq:3.9}
 F: C\o C\lr C\o C,~~c\o d\mapsto Q(c)\o d+c\o T(d)
 \end{equation}
 is a weak copseudotwistor.
\end{pro}

\begin{proof} Define $K$-linear map $\widetilde{F}: C\o C\o C\lr C\o C\o C$,
 \begin{equation}\label{eq:3.10}
 \widetilde{F}(c\o d\o e)=Q(c)\o Q(d)\o e+Q(c)\o d\o T(e)+c\o T(d)\o T(e),
 \end{equation}
 where $c, d, e\in C$. Then we  obtain
 \begin{eqnarray*}
 (\widetilde{F}\ci (\D\o id))(c\o d)&\stackrel{}{=}&Q(c_1)\o Q(c_2)\o d+Q(c_1)\o c_2\o T(d)+c_1\o T(c_2)\o T(d)\\
 &\stackrel{(\ref{eq:3.1})}{=}&Q(Q(c)_1)\o Q(c)_2\o d+Q(c)_1\o T(Q(c)_2)\o d\\
 &&+Q(c_1)\o c_2\o T(d)+c_1\o T(c_2)\o T(d)\\
 &\stackrel{}{=}&((F\o id)\ci (\D\o id)\ci F)(c\o d)
 \end{eqnarray*}
 and
 \begin{eqnarray*}
 (\widetilde{F}\ci (id\o \D))(c\o d)&\stackrel{}{=}&Q(c)\o Q(d_1)\o d_2+Q(c)\o d_1\o T(d_2)+c\o T(d_1)\o T(d_2)\\
 &\stackrel{(\ref{eq:3.1})}{=}&Q(c)\o Q(d_1)\o d_2+Q(c)\o d_1\o T(d_2)\\
 &&+c\o Q(T(d)_1)\o T(d)_2+c\o T(d)_1\o T(T(d)_2)\\
 &\stackrel{}{=}&((id\o F)\ci (id\o \D)\ci F)(c\o d).
 \end{eqnarray*}
 Thus Eq.(\ref{eq:3.8}) holds. This finishes the proof. \end{proof}

\begin{rmk}\label{rmk:3.10}
 (1) The coproduct $\D_{*}$ defined by Eq.(\ref{eq:3.6}) is simply equal to $F\ci \D$, where $F$ is given in Eq.(\ref{eq:3.9}).

 (2) If $C$ is a coalgebra satisfying that for a map $f: C\lr C$ such that for all $c, d\in C$, $c\o f(d)=f(c)\o d=0$ implies that $f=0$,  then $F$ given in Eq.(\ref{eq:3.9}) is a weak copseudotwistor with the companion defined in  Eq.(\ref{eq:3.10}) if and only if $(C, Q, T)$ is a Rota-Baxter cosystem.
\end{rmk}

\begin{defi}\label{de:3.11}
 {\rm  A quintuple $(H, R, S, Q, T)$ consisting of a bialgebra $H$ and four $K$-linear maps $R, S, Q, T: H\lr H$ is  called a {\bf Rota-Baxter bisystem} if $(H, R, S)$ is a Rota-Baxter system and $(H, Q, T)$ is a Rota-Baxter cosystem.}
\end{defi}

\begin{thm}\label{thm:3.12}
 Let $\l, \g\in K$ and $H$ a bialgebra (not necessarily unital and counital). If $(H, R, Q)$ is a Rota-Baxter bialgebra of weight $(\l, \g)$, then $(H, R, R+\l id, Q, Q+\g id)$, $(H, R+\l id, R, Q, Q+\g id)$, $(H, R, R+\l id, Q+\g id, Q)$, $(H, R+\l id, R, Q+\g id, Q)$ are Rota-Baxter bisystems.
\end{thm}

\begin{proof} It is obvious by \cite[Lemma 2.2]{Br} and Proposition \ref{pro:3.3}. \end{proof}

\begin{pro}\label{pro:3.13}
 Let $H$ be a bialgebra, $R: H\lr H$ a left $H$-linear map, $S: H\lr H$  a right $H$-linear map, $Q: H\lr H$ a left $H$-colinear map and  $T: H\lr H$ a right $H$-colinear map. Then $(H, R, S, Q, T)$ is a Rota-Baxter bisystem if and only if, for all $x, y\in H$,
 $$
 a(R\ci S)(b)=0=(S\ci R)(a)b
 $$
 and
 $$
 x_1\o (T\ci Q)(x_2)=(Q\ci T)(x_1)\o x_2=0.
 $$
 In particular, if $H$ is non-degenerate as algebra and as coalgebra, respectively, then $(H, R, S, Q, T)$ is a Rota-Baxter bisystem (with $R$ left and $S$ right $H$-linear, $Q$ left and $T$ right $H$-colinear) if and only if $R$ and $S$, $Q$ and $T$ satisfy the orthogonality condition, respectively,
 $$
 R\ci S=S\ci R=0,~~~~~T\ci Q=Q\ci T=0.
 $$
\end{pro}

\begin{proof} It is obvious by \cite[Lemma 2.3]{Br} and Proposition \ref{pro:3.4}. \end{proof}

\begin{ex}\label{ex:3.14}
 Let $C$ be a bialgebra and $H$ a Hopf algebra with the antipode $S$. Suppose that there are two bialgebra maps: $i: H\lr C$ and $\pi: C\lr H$ such that $\pi\circ i=id_{H}$, i.e., $C$ is a bialgebra with a projection (see \cite{Ra12}). Set $\Pi=id_{C}\star(i\circ S\circ \pi)$, where $\star$ is the convolution product on End($C$). The right $H$-Hopf module structure is given by the following:
 $$
 c\cdot h=ci(h),
 $$
 $$
 \rho_{R}(c)=c_1\o \pi(c_2),
 $$
 for all $c\in C$ and $h\in H$. Then $(C, \Pi, \Pi)$ is a Rota-Baxter bialgebra of weight $(-1, -1)$, and $(C, \Pi, \Pi-id, \Pi, \Pi-id)$, $(C, \Pi-id, \Pi, \Pi, \Pi-id)$, $(C, \Pi, \Pi-id, \Pi-id, \Pi)$, $(C, \Pi-id, \Pi, \Pi-id, \Pi)$ are Rota-Baxter bisystems.
\end{ex}

\begin{proof} It is a direct consequence by \cite[Example 5.3]{ML} and Theorem \ref{thm:3.12}.\end{proof}

\begin{ex}\label{ex:3.15}
 Let $H$ be a bialgebra, $f: H\lr H$ a multiplicative map, $g: H\lr H$ a comultiplicative map and  $R: C\lr C$ a $K$-linear map such that for all $x, y\in H$,
 $$
 R(x)R(y)=R(R(x)y+xR^f(y))
 $$
 where $R^f=R\ci f$ and $Q: H\lr H$ a $K$-linear map such that for all $x\in H$, (3.3) holds. Then $(H, R, R^f, Q, Q^g)$ is a Rota-Baxter bisystem.
\end{ex}

\begin{proof} It is a direct consequence by \cite[Lemma 4.2]{Br} and Example \ref{ex:3.2}.\end{proof}

\begin{ex}\label{ex:3.16}
 Let $\{x, y, z\}$ be a basis of a 3-dimensional vector space $H$ over $K$. We define multiplication and comultiplication over $H$:
 $$
 x^{2}=x,~~ y^{2}=y,~~ z^{2}=0
 $$
 $$
 xy=yx=y, ~~yz=z,
 $$
 $$
 xz=zx=z, ~~zy=0,
 $$
 $$
 \D(x)=x\o x, ~~\D(y)=y\o y, ~~\D(z)=z\o z.
 $$
 It is not hard to see that $H$ is a bialgebra with the above constructions. Let $R, S, Q, T$ be operators defined with respect to the basis $\{x, y, z\}$ by
 $$\begin{array}{ccc}
   R(x)=\l_1 z, & R(y)=\l_2  z,  & R(z)=0,   \\
   S(x)=-\l_3  y, & S(y)=0,   & S(z)=0,    \\
   Q(x)=\l_4  z, & Q(y)=\l_4 z,  & Q(z)=0,  \\
   T(x)=\l_4(x+z), & T(y)=\l_4(y+z),  & T(z)=\l_4 z,
 \end{array}
 $$
 where $\l_i\in K, i=1,2,3,4$. Then $(H, R, S, Q, T)$ is a Rota-Baxter bisystem.
\end{ex}

\section{Covariant bialgebras} \label{sec:covarbialg}
 In this section, we redefine  covariant bialgebras via coderivations. Moreover, we consider  coassociative Yang-Baxter pairs and coassociative Yang-Baxter equation.

\begin{defi}\label{de:4.1}
 {\rm Let $C$ be a coassociative coalgebra. A coassociative Yang-Baxter pair (abbr. CYBP) on $C$ is a pair of elements $\s, \tau\in (C\o C)^*$ that satisfy the following equations
 \begin{eqnarray}
 &&\s(c_1, e)\s(c_2, d)-\s(c, d_1)\s(d_2, e)+\tau(d, e_1)\s(c, e_2)=0, \label{eq:4.1}\\
 &&\tau(c_1, e)\s(c_2, d)-\tau(c, d_1)\tau(d_2, e)+\tau(d, e_1)\tau(c, e_2)=0, \label{eq:4.2}
 \end{eqnarray}
 where $\forall~c, d, e\in C$.}
\end{defi}

\begin{pro}\label{pro:4.2}
 Let $f: C\lr D$ be a coalgebra map. If $(\s, \tau)$ is a CYBP on $D$, then
 \begin{equation}\label{eq:4.3}
 \s^f=\s\ci (f\o f)~~~\hbox{and}~~~\tau^f=\tau\ci (f\o f)
 \end{equation}
 form a CYBP on $C$.
\end{pro}

\begin{proof}
 We compute Eq.(\ref{eq:4.1}) for $(\s^f, \tau^f)$ as follows.
 \begin{eqnarray*}
 &&\s^f(c_1, e)\s^f(c_2, d)-\s^f(c, d_1)\s^f(d_2, e)+\tau^f(d, e_1)\s^f(c, e_2)\\
 &&\stackrel{}{=}\s(f(c)_1, f(e))\s(f(c)_2, f(d))-\s(f(c), f(d)_1)\s(f(d)_2, f(e))\\
 &&+\tau(f(d), f(e)_1)\s(f(c), f(e)_2)~~\hbox{by~}f~\hbox{is coalgebra map}\\
 &&\stackrel{(\ref{eq:4.1})}{=}0.
 \end{eqnarray*}
 The proof for Eq.(\ref{eq:4.2}) is similar. \end{proof}

\begin{ex}\label{ex:4.3}
 (1) The pair $(\s, \s)$ is a CYBP if and only if $\s$ is a solution of the coassociative Yang-Baxter equation
 $$
 \s(c_1, e)\s(c_2, d)-\s(c, d_1)\s(d_2, e)+\s(d, e_1)\s(c, e_2)=0.
 $$

 (2) If for all $c, d, e\in C$, $\s\in (C\o C)^*$ satisfies
 \begin{equation}\label{eq:4.4}
 \s(c, d_1)\s(d_2, e)=\s(d, e_1)\s(c, e_2)=\s(c_1, e)\s(c_2, d),
 \end{equation}
 the dual case of Frobenius-separability equation in \cite{CMZ}, then $(\s, 0)$ and $(0, \s)$ are CYBPs.

 (3) If for all $c, d, e\in C$, $\s, \tau\in (C\o C)^*$ satisfies (4.4), such that
 \begin{equation}\label{eq:4.5}
 \tau(d, e_1)\s(c, e_2)=\tau(c_1, e)\s(c_2, d)=0,
 \end{equation}
 then $(\s, \tau)$ is a CYBP.

 (4) If for all $c, d, e\in C$, $\s, \tau\in (C\o C)^*$ satisfies
 \begin{equation}\label{eq:4.6}
 \s(c_1, e)\s(c_2, d)=\s(c, d_1)\s(d_2, e),~~\tau(c, d_1)\tau(d_2, e)=\tau(d, e_1)\tau(c, e_2),
 \end{equation}
 then $(\s, 0)$ and $(0, \tau)$ are CYBPs. If in addition, $\s, \tau$ satisfy (4.5), then $(\s, \tau)$ is a CYBP.

 (5) Let $C$ be a coalgebra. Suppose that $\xi, \zeta\in C^*$ such that for all $c\in C$, $\xi(c_1)\xi(c_2)=0$ and $\xi(c_1)\zeta(c_2)=\zeta(c_1)\xi(c_2)=0$, then $(\s=\xi\o \zeta, \tau=\zeta\o \xi)$ is a CYBP.
\end{ex}

\begin{pro}\label{pro:4.4}
 Let $(\s, \tau)$ be a CYBP on $C$. Define $Q, T: C\lr C$ by
 \begin{equation}\label{eq:4.7}
 Q(c)=\s(c_1, c_3)c_2,~~~T(c)=\tau(c_1, c_3)c_2.
 \end{equation}
 Then $(C, Q, T)$ is a Rota-Baxter cosystem associated to $(\s, \tau)$.

 Furthermore, if $f: C\lr D$ is a coalgebra map and $(\s^f, \tau^f)$ is the CYBP induced by $f$ as in Eq.(\ref{eq:4.3}), then $f$ is a morphism of Rota-Baxter cosystems from the cosystem associated to $(\s^f, \tau^f)$ to the cosystem associated to $(\s, \tau)$.
\end{pro}

\begin{proof}
 We only check that Eq.(\ref{eq:3.1}) holds as follows. For all $c\in C$, we have
 \begin{eqnarray*}
 \hbox{LHS of Eq.(\ref{eq:3.1})}&\stackrel{}{=}&\s(c_{1}, c_{31})c_{2}\o \s(c_{32}, c_{5})c_{4}\\
 &\stackrel{(\ref{eq:4.1})}{=}&\s(c_{11}, c_{5})\s(c_{12}, c_{3})c_{2}\o c_{4}+\s(c_{3}, c_{51})\s(c_{1}, c_{52})c_{2}\o c_{4}\\
 &\stackrel{ }{=}&\s(c_{1}, c_{3})\s(c_{211}, c_{213})c_{212}\o c_{22}+\s(c_{1}, c_{3})\s(c_{221}, c_{223})c_{21}\o c_{222}\\
 &\stackrel{ }{=}&\hbox{RHS of Eq.(\ref{eq:3.1})},
 \end{eqnarray*}
 finishing the proof. \end{proof}

\begin{rmk}\label{rmk:4.5}
 By \cite[Proposition 3.4]{Br} and Proposition \ref{pro:4.4}, we  obtain:\\ Let $H$ be a bialgebra. Suppose that $(r, s)$ is an associative Yang-Baxter pair on $H$ (see \cite{Br}) and $(\s, \tau)$ is a CYBP on $H$, then $(H, R, S, Q, T)$ is a Rota-Baxter bisystem.
\end{rmk}

\begin{ex}\label{ex:4.6}
 (1) The Rota-Baxter cosystem associated to $(\s, \tau)$ satisfying Eqs.(\ref{eq:4.5}) and (\ref{eq:4.6}) in Example \ref{ex:4.3} (4) will satisfy the separated equations
 $$
 Q(c_1)\o Q(c_2)=Q(Q(c)_1)\o Q(c)_2,~T(c_1)\o T(c_2)=T(c)_1\o T(T(c)_2),
 $$
 and
 $$
 Q(c)_1\o T(Q(c)_2)=Q(T(c)_1)\o T(c)_2=0.
 $$
 For example,
 $$
 Q(c)_1\o T(Q(c)_2)=\s(c_1, c_3)\tau(c_{221}, c_{223})c_{21}\o T(c_{222})\stackrel{(\ref{eq:4.5})}{=}0
 $$

 (2) The Rota-Baxter cosystem associated to $(\s, \tau)$ in Example \ref{ex:4.3} (5) is $(C, Q, T)$, where $Q(c)=\xi(c_1)c_2\zeta(c_3)$ and $T(c)=\zeta(c_1)c_2\xi(c_3)$.
\end{ex}

\begin{defi}\label{de:4.7}
 {\rm Let $C$ be a coassociative coalgebra and $\d_1, \d_2: C\o C\lr C$ (write $\d_{i}(c\o d)=c\di_i d, i=1, 2$) be coderivations (i.e. $(c\di_{i} d)_1\o (c\di_{i} d)_2=c_1\o c_2\di_{i} d+c\di_{i} d_1\o d_2,  i=1, 2)$ (see \cite{Do}).

 (1) If $M$ is a right $C$-comodule, then a $K$-linear map $\nabla_r: M\o C\lr M$ (write $\nabla_r(m\o c)=m\tr c$) is called a {\bf right covariant coderivation with respect to $\d_1$} if
 \begin{equation}\label{eq:4.8}
 (m\tr c)_{(0)}\o (m\tr c)_{(1)}=m\tr c_1\o c_2+m_{(0)}\o m_{(1)}\di_1 c
 \end{equation}
 for all $m\in M$ and $c\in C$.

 (2) If $M$ is a left $C$-comodule, then a $K$-linear map $\nabla_l: C\o M\lr M$ (write $\nabla_r(c\o m)=c\tl m$) is called a {\bf left covariant coderivation with respect to $\d_2$} if
 \begin{equation}\label{eq:4.9}
 (c\tl m)_{(-1)}\o (c\tl m)_{(0)}=c_1\o c_2\tl m +c\di_2 m_{(-1)}\o  m_{(0)}
 \end{equation}
 for all $m\in M$ and $c\in C$.

 (3) A $K$-linear map $\nabla: C\o C\lr C$ (write $\nabla(c\o d)=c\cdot d$) is called a {\bf covariant coderivation with respect to $(\d_1, \d_2)$} if it is a right covariant coderivation with respect to $\d_1$ and left covariant coderivation with respect to $\d_2$, i.e.
 \begin{equation}\label{eq:4.10}
 (c\cdot d)_{1}\o (c\cdot d)_{2}=c\cdot d_1\o d_2+c_{1}\o c_{2}\di_1 d=c_1\o c_2\cdot d +c\di_2 d_1\o  d_{2}
 \end{equation}
 for all $m\in M$ and $c\in C$.} \end{defi}

\begin{rmk}\label{rmk:4.8}
 (1) Obviously, any coderivation $\d: C\o C\lr C$ is a covariant coderivation with respect to $(\d, \d)$.

 (2) A covariant coderivation with respect to $(0, 0)$ is the same as a $C$-bicomodule map $C\o C\lr C$, where $\r_r(c\o d)=c\o d_1\o d_2$ and $\r_l(c\o d)=c_1\o c_2\o d$.
\end{rmk}

 In the following,  we introduce a new algebraic structure using  coderivations and covariant coderivations, close to covariant bialgebra which we still call covariant bialgebra.

\begin{defi}\label{de:4.9}
 {\rm  A {\bf covariant bialgebra} is a quadruple $(C, \d_1, \d_2, \mu)$, such that

 (1) $C$ is a coassociative coalgebra,

 (2) $\d_1, \d_2:C\o C\lr C$ are coderivations,

 (3) $(C, \mu)$ is an associative algebra such that $\mu$ is a covariant coderivation with respect to $(\d_1, \d_2)$.

 If $C$ has a counit, then a covariant bialgebra $(C, \d_1, \d_2, \mu)$ is said to be {\bf counital} if moreover we have  $\v(cd)=\v(c)\v(d)$, here $cd=\mu(c\o d)$.

 A {\bf morphism of covariant bialgebras} is a $K$-linear map that is both an algebra and a coalgebra map.}
\end{defi}

\begin{ex}\label{ex:4.10}
 $(C, \mu, \mu, \mu)$ is a covariant bialgebra if and only if $(c, \D)$ is an infinitesimal bialgebra \cite{Ag}, i.e. an algebra equipped with a coassociative comultiplication which is also a derivation, because infinitesimal bialgebras are  self-dual.
\end{ex}

\begin{thm}\label{thm:4.11}
 Let $C$ be a counital coalgebra, and let $\d_1, \d_2:C\o C\lr C$ be coderivations. Then there exists an associative covariant coderivation $\mu: C\o C\lr C$ (write $\mu(c\o d)=cd$) with respect to $(\d_1, \d_2)$ if and only if there exists $u \in (C\o C)^*$ such that, for all $c, d, e\in C$
 \begin{equation}\label{eq:4.11}
 c\di_1 d-c\di_2 d=c_1u(c_2, d)-u(c, d_1)d_2,
 \end{equation}
 \begin{equation}\label{eq:4.12}
 (c\di_1 d)\di_1 e-c\di_1 (d\di_1 e)=u(d, e_1)c\di_1 e_2,
 \end{equation}
 \begin{equation}\label{eq:4.13}
 u(c\di_1 d, e)-u(c, d\di_1 e)=u(d, e_1)u(c, e_2)-u(c, d_1)u(d_2, e).
 \end{equation}
 In this case,
 \begin{equation}\label{eq:4.14}
 cd=u(c, d_1)d_2+c\di_1 d=c_1u(c_2, d)+c\di_2 d.
 \end{equation}
\end{thm}

\begin{proof}
 ($\Longrightarrow$) Set $u=\v\ci \mu\in (C\o C)^*$. Apply $\v\o id$ and $id\o \v$ to Eq.(\ref{eq:4.10}) for $\mu$, respectively, we  get Eq.(\ref{eq:4.14}),  then Eq.(\ref{eq:4.11}) holds.

 By the associativity of $\mu$, we have
 \begin{equation}\label{eq:4.15}
 \mu \ci (\mu\o id)=\mu \ci (id\o \mu).
 \end{equation}
 Then by Eq.(\ref{eq:4.14}),
 $$
 \mu(u(c, d_1)d_2\o e+c\di_1 d\o e)=\mu(c\o u(d, e_1)e_2+c\o d\di_1 e).
 $$
 Apply $\v$ to both sides of above equality and by the definition of $u$, one  obtain Eq.(\ref{eq:4.13}).

 Since
 $$
 \hbox{LHS of Eq.(\ref{eq:4.15})}\stackrel{(\ref{eq:4.14})}{=}u(c, d_1)u(d_2, e_1)e_2+u(c, d_1)d_2\di_1 e+u(c\di_1 d, e_1)e_2+(c\di_1 d)\di_1 e,
 $$
 and
 $$
 \hbox{RHS of Eq.(\ref{eq:4.15})}\stackrel{(\ref{eq:4.14})}{=}u(d, e_1)u(c, e_{21})e_{22}+u(d, e_1)c\di_1 e_2+u(c, (d\di_1 e)_1)(d\di_1 e)_2+c\di_1 (d\di_1 e),
 $$
 \begin{eqnarray*}
 (c\di_1 d)\di_1 e-c\di_1 (d\di_1 e)
 &\stackrel{}{=}&((u(c, d\di_1 e_1)-u(c\di_1 d, e_1))-(u(c, d_1)u(d_2, e_1)\\
 &&-u(d, e_{11})u(c, e_{12}))e_2+u(d, e_1)c\di_1 e_2~~ (\hbox{by coderivation})\\
 &\stackrel{(\ref{eq:4.13})}{=}&u(d, e_1)c\di_1 e_2,
 \end{eqnarray*}
 so Eq.(\ref{eq:4.12}) is proved.

 ($\Longleftarrow$) By Eq.(\ref{eq:4.14}), we have
 \begin{eqnarray*}
 (cd)_1\o (cd)_2
 &\stackrel{}{=}&u(c, d_1)d_2\o d_3+(c\di_1 d)_1\o (c\di_1 d)_2\\
 &=&u(c, d_1)d_2\o d_3+c_1\o c_2\di_{1} d+c\di_{1} d_1\o d_2~~ (\hbox{by coderivation})\\
 &=&cd_1\o d_2+c_1\o c_2\di_{1} d.
 \end{eqnarray*}
 Similarly, we  get $(cd)_1\o (cd)_2=c_1\o c_2d+c\di_{2} d_1\o d_2$. The rest is obvious.    \end{proof}

\begin{cor}\label{cor:4.12}
 A quadruple $(C, \d_1, \d_2, \mu)$ consisting of a counital coalgebra $C$, coderivations $\d_1, \d_2: C\o C\lr C$ and a $K$-linear map $\mu: C\o C\lr C$ is a counital covariant bialgebra if and only if, for all $c, d, e\in C$,
 \begin{equation}\label{eq:4.16}
 c\di_1 d-c\di_2 d=c\v(d)-\v(c)d,
 \end{equation}
 \begin{equation}\label{eq:4.17}
 (c\di_1 d)\di_1 e-c\di_1 (d\di_1 e)=\v(d)c\di_1 e.
 \end{equation}
 In this case,
 \begin{equation}\label{eq:4.18}
 cd=\v(c)d+c\di_1 d.
 \end{equation}
\end{cor}

\begin{proof}
 Assume $(C, \d_1, \d_2, \mu)$  to be a counital covariant bialgebra, then $u=\v\o \v$, and the rest is obvious.  \end{proof}

\begin{thm}\label{thm:4.13}
 Let $C$ be a coalgebra and $\s, \tau\in C^*$. Define $K$-linear maps $\d_{\s}, \d_{\tau}, \mu: C\o C\lr C$ by
 \begin{eqnarray}
 \d_{\s}(c\o d)&=&c_1\s(c_2, d)-\s(c, d_1)d_2,  \nonumber\\
 \d_{\tau}(c\o d)&=&c_1\tau(c_2, d)-\tau(c, d_1)d_2, \label{eq:4.19}\\
 \mu(c\o d)&=&c_1\s(c_2, d)-\tau(c, d_1)d_2.  \nonumber
 \end{eqnarray}
 Then $(C, \d_{\s}, \d_{\tau}, \mu)$ is a covariant bialgebra if and only if, for all $c, d, e\in C$,
 \begin{eqnarray}
 &&c_1(\s(c_{21}, e)\s(c_{22}, d)-\s(c_2, d_1)\s(d_2, e)+\tau(d, e_1)\s(c_2, e_2))\nonumber\\
 &&~~=(\tau(c_1, e_1)\s(c_2, d)-\tau(c, d_1)\tau(d_2, e_1)+\tau(d, e_{11})\tau(c, e_{12}))e_2.   \label{eq:4.20}
 \end{eqnarray}
\end{thm}

\begin{proof}
 It is straightforward to check that $\d_{\s}, \d_{\tau}$ are coderivations, and $\mu$ is a covariant coderivation with respect to $\d_{\s}, \d_{\tau}$. Thus in order to finish the proof we only need to prove that $(C, \mu)$ is an associative algebra if and only if Eq.(\ref{eq:4.20}) holds. While
 \begin{eqnarray*}
 (\mu \ci (id\o \mu))(c\o d\o e)&=&c_1\s(c_2, d_1)\s(d_2, e)-c_1\tau(d, e_1)\s(c_2, e_2)\\
 &&+\tau(d, e_{11})\tau(c, e_{12})e_2-\tau(c, d_1)d_2\s(d_3, e).
 \end{eqnarray*}
 and
 \begin{eqnarray*}
 (\mu \ci (\mu\o id))(c\o d\o e)&=&c_1\s(c_{21}, e)\s(c_{22}, d)-(\tau(c_1, e_1)\s(c_2, d)e_2\\
 &&+\tau(c, d_1)\tau(d_2, e_1)e_2-\tau(c, d_1)d_2\s(d_3, e),
 \end{eqnarray*}
 therefore we complete the proof.   \end{proof}

\begin{rmk}\label{rmk:4.14}
 If $C$ is a non-degenerate coalgebra, then the form of $\mu$ in Eq.(\ref{eq:4.19}) specifies the forms of $\d_{\s}$ and $\d_{\tau}$. Indeed, let us suppose that $\mu(c\o d)=c_1\s(c_2, d)-\tau(c, d_1)d_2$ is a $(\d_1, \d_2)$-covariant coderivation. Then, for all $c, d\in C$,
 $$
 c_1\o c_2\di_1 d=c_1\o(c_{21}\s(c_{22}, d)-\s(c_2, d_1)d_2)
 $$
 and
 $$
 c\di_2 d_1\o d_2=(c_1\tau(c_2, d_1)-\tau(c, d_{11})d_{12})\o d_2.
 $$
 Since $C$ is non-degenerate, these equations imply that $\d_1=\d_{\s}$ and $\d_2=\d_{\tau}$.
\end{rmk}

\section{Coquasitriangular covariant bialgebras and infinitesimal bialgebras}
\label{sec:CoquasitriangCovBialgInfinBialg}
In this section, we study  coquasitriangular covariant bialgebras generalizing coquasitriangular infinitesimal bialgebras, and  corresponding to our covariant bialgebras.
\begin{pro}\label{cor:5.1}  Let $C$ be a coalgebra and $\s, \tau\in C^*$.
 If $(\s, \tau)$ is a CYBP on $C$ and $\d_{\s}, \d_{\tau}, \mu$ are defined by Eq.(\ref{eq:4.19}), then $(C, \d_{\s}, \d_{\tau}, \mu)$ is a covariant bialgebra. In this case, $(C, \d_{\s}, \d_{\tau}, \mu)$ is called a {\bf coquasitriangular covariant bialgebra}.

 Furthermore, if $f: C\lr D$ is a coalgebra map and $(\s^f, \tau^f)$ is the CYBP induced by $f$ as in Eq.(\ref{eq:4.3}), then $f$ is a morphism of quasitriangular covariant bialgebras from the one associated to $(\s^f, \tau^f)$ to the one associated to $(\s, \tau)$.
\end{pro}

\begin{proof} The first statement follows immediately from Theorem \ref{thm:4.13}. The second one can be proved directly.  \end{proof}

\begin{rmk}\label{rmk:5.2}
 When $\s=\tau$ in Proposition \ref{cor:5.1}, $\d_{\s}=\d_{\tau}=\mu$, then $(C, \d_{\s}, \d_{\tau}, \mu)$ is an infinitesimal bialgebra. In this case, we call $(C, \d_{\s}, \d_{\tau}, \mu)$ (abbr. $(C, \d_{\s})$) a {\bf coquasitriangular infinitesimal bialgebra}, which will be studied later.
\end{rmk}

\begin{pro}\label{pro:5.3}
 Let $C$ be a coalgebra, $\s, \tau\in (C\o C)^*$, $\d_{\s}, \d_{\tau}, \mu$ are defined by Eq.(\ref{eq:4.19}). Then $(C, \d_{\s}, \d_{\tau}, \mu)$ is a coquasitriangular covariant bialgebra if and only if for all $c, d, e\in C$,
 \begin{equation}\label{eq:5.1}
 \s(c, de)=\s(c_1, e)\s(c_2, d),~~\tau(cd, e)=-\tau(d, e_1)\tau(c, e_2).
 \end{equation}
\end{pro}

\begin{proof} It is straightforward to prove that Eq.(\ref{eq:5.1}) is equivalent to Eqs.(\ref{eq:4.1}) and (\ref{eq:4.2}), i.e. $(\s, \tau)$ is a CYBP.  \end{proof}


\begin{pro}\label{pro:5.5}
 Counital coquasitriangular covariant bialgebra structures on $C$ are determined by $\s\in (C\o C)^*$ such that for all $c, d, e\in C$,
 \begin{equation}\label{eq:5.2}
 \s(c, de)=\s(c_1, e)\s(c_2, d),~~\tau(cd, e)=-\s(d, e_1)\s(c, e_2)+\v(c)\s(d, e)+\s(c, e)\v(d),
 \end{equation}
 where $cd=\v(c)d+c_1\s(c_2, d)-\s(c, d_1)d_2$.
\end{pro}

 \begin{proof}
 A coquasitriangular covariant bialgebra $(C, \d_{\s}, \d_{\tau}, \mu)$ is counital if and only if $\s(c, d)-\tau(c, d)=\v(c)\v(d)$, i.e., $\tau(c, d)=\s(c, d)-\v(c)\v(d)$. Thus we  complete the proof by replacing $\tau$ by $\s-\v\o \v$ in Eq.(\ref{eq:5.1}). \end{proof}

\begin{ex}\label{ex:5.6}
 Let $C$ be a counital coalgebra and $\a\in C^*$ such that $\a(c_1)\a(c_2)=\a(c)$, $\forall~c\in C$, and $k\in \{0, 1\}\subset K$. Then
 $\s^{\a}(c\o d)=k\v(c)\a(d)+(1-k)\a(c)\v(d)$ satisfies Eq.(\ref{eq:5.2}). Thus $(\s^{\a}, \tau^{\a})$ is a CYBP, where
 $$
 \tau^{\a}(c\o d)=\s^{\a}(c\o d)-\v(c)\v(d)=k\v(c)\a(d)+(1-k)\a(c)\v(d)-\v(c)\v(d),
 $$
 and
 $$
 cd=\mu^{\a}(c\o d)=\v(c)d+k(c\a(d)-\v(c)\a(d_1)d_2)+(1-k)(c_1\a(c_2)\v(d)-\a(c)\v(d)).
 $$
 Then $(C, \s^{\a}, \tau^{\a}, \mu^{\a})$ is a coquasitriangular covariant bialgebra.

 By Proposition \ref{pro:4.4}, $(C, Q^{\s^{\a}}, T^{\tau^{\a}})$ is a Rota-Baxter cosystem associated to $(\s^{\a}, \tau^{\a})$, where
 $$
 Q^{\s^{\a}}(c)=kc_1\a(c_2)+(1-k)\a(c_1)c_2, ~T^{\tau^{\a}}(c)=kc_1\a(c_2)+(1-k)\a(c_1)c_2-c.
 $$

 By Proposition \ref{pro:3.6}, $(C, \D_{*})$ is a coassociative coalgebra, where
 $$
 \D_{*}(c)=k(c_1\a(c_2)\o c_3+c_1\o c_2\a(c_3))+(1-k)(\a(c_1)c_2\o c_3+c_1\o \a(c_2)c_3)-c_1\o c_2.
 $$
 And $(C, \D_{\bullet})$ is a pre-Lie coalgebra, where
 $$
 \D_{\bullet}(c)=k(c_1\a(c_2)\o c_3-c_2\o c_1\a(c_3))+(1-k)(\a(c_1)c_2\o c_3-c_3\o \a(c_2)c_1)-c_2\o c_1.
 $$

 Define $\d(c\o d)=\v(c)d-c\v(d)$. Then $(C, 0, -\d, \mu^0)$ and $(C, \s, 0, \mu^{\v})$ are counital covariant bialgebras associated to the coderivation $\d$.
\end{ex}

\begin{defi}\label{de:5.7}
 {\rm A {\bf coquasitriangular infinitesimal bialgebra} is a pair $(C, \s)$ where $C$ is a coassociative coalgebra and $\s\in (C\o C)^*$ is a solution to
 \begin{equation}\label{eq:5.3}
 \s(c_1, e)\s(c_2, d)-\s(c, d_1)\s(d_2, e)+\s(d, e_1)\s(c, e_2)=0
 \end{equation}
 for all $c, d, e\in C$.

 We call Eq.(\ref{eq:5.3}) a {\bf coassociative Yang-Baxter equation (abbr. CYBA) in $C$}.}
\end{defi}

\begin{rmk}\label{rmk:5.8}The linear form
 $\s\in (C\o C)^*$ is a solution to CAYB if and only if $(C, \mu_{\s})$ is an infinitesimal bialgebra, where $\mu_{\s}(c\ d)=c_1\s(c_2, d)-\s(c, d_1)d_2$, $\forall c, d\in C$.
\end{rmk}

\begin{ex}\label{ex:5.9}
 Let $C$ be any counital coalgebra possessing an element $\tau\in C^*$ such that $\tau(c_1)\tau(c_2)=0$. Then $\s=\v\o \tau$ satisfies CAYB and the corresponding infinitesimal bialgebra structure is
 $$
 \mu_{\s}(c\o d)=c\tau(d)-\v(c)\tau(d_1)d_2, \forall c, d\in C.
 $$
\end{ex}

\begin{pro}\label{pro:5.10}
 A principal coderivation $\mu_{\s}: C\o C\lr C$ is associative if and only if, for all $c, d, e\in C$,
 \begin{eqnarray*}
 &&c_1(\s(c_{21}, e)\s(c_{22}, d)-\s(c_2, d_1)\s(d_2, e)+\s(d, e_1)\s(c_2, e_2))\\
 &&~~=(\s(c_1, e_1)\s(c_2, d)-\s(c, d_1)\s(d_2, e_1)+\s(d, e_{11})\s(c, e_{12}))e_2.
 \end{eqnarray*}
\end{pro}

\begin{pro}\label{pro:5.11}The pair
 $(C, \s)$ is a coquasitriangular infinitesimal bialgebra if and only if for all $c, d, e\in C$,
 \begin{equation}\label{eq:5.4}
 \s(cd, e)=-\s(d, e_1)\s(c, e_2),
 \end{equation}
 \begin{equation}\label{eq:5.5}
 \s(c, de)=\s(c_1, e)\s(c_2, d).
 \end{equation}
\end{pro}

\begin{proof} Let $\s=\tau$ in Proposition \ref{pro:5.3}.  \end{proof}

 If $(C, \mu, \D)$ is a finite dimensional infinitesimal bialgebra, then $C'=(C^*, \D^{*op}, -\mu^{*cop})$ and ${'C}=(C^*, -\D^{*op}, \mu^{*cop})$ are infinitesimal bialgebras (see \cite{Ag}).

\begin{pro}\label{pro:5.12}
 Let $(C, \s)$ be a finite dimensional coquasitriangular infinitesimal bialgebra. Then the maps
 \begin{equation}\label{eq:5.6}
 \zeta_{\s}: C\lr C',~~\zeta_{\s}(c)=\sum e^i\s(e_i, c),
 \end{equation}
 \begin{equation}\label{eq:5.7}
 \xi_{\s}: C\lr {'C},~~\xi_{\s}(c)=\sum e^i\s(c, e_i)
 \end{equation}
 where $c\in C$, $\{e_i\}$ is the basis of $C$ and $\{e^i\}$ is the dual basis, are morphisms of infinitesimal bialgebras.
\end{pro}

\begin{proof}
 We note that for all $b, c\in C$, $ \langle \zeta_{\s}(c), b \rangle =\s(b, c)$ and $ \langle \xi_{\s}(c), b \rangle =\s(c, b)$. For all $a, b, c\in C$,
 \begin{eqnarray*}
  \langle \zeta_{\s}(ab), c \rangle
 &=&\s(c, ab)\stackrel{(\ref{eq:5.5})}{=}\s(c_1, b)\s(c_2, a)\\
 &=& \langle \zeta_{\s}(b), c_1 \rangle  \langle \zeta_{\s}(a), c_2 \rangle = \langle \zeta_{\s}(a)\zeta_{\s}(b), c \rangle ,
 \end{eqnarray*}
 \begin{eqnarray*}
  \langle \zeta_{\s}(a)_1\o \zeta_{\s}(a)_2, b\o c \rangle
 &=&- \langle \zeta_{\s}(a), cb \rangle =-\s(cb, a)\\
 &\stackrel{(\ref{eq:5.4})}{=}&\s(b, a_1)\s(c, a_2)= \langle \zeta_{\s}(a_1), b \rangle  \langle \zeta_{\s}(a_2), c \rangle \\
 &=& \langle \zeta_{\s}(a_1)\o \zeta_{\s}(a_2), b\o c \rangle ,
 \end{eqnarray*}
 \begin{eqnarray*}
  \langle \xi_{\s}(ab), c \rangle
 &=&\s(ab, c)\stackrel{(\ref{eq:5.4})}{=}-\s(b, c_1)\s(a, c_2)\\
 &=&- \langle \xi_{\s}(b), c_1 \rangle  \langle \xi_{\s}(a), c_2 \rangle = \langle \xi_{\s}(a)\xi_{\s}(b), c \rangle ,
 \end{eqnarray*}
 \begin{eqnarray*}
  \langle \xi_{\s}(a)_1\o \xi_{\s}(a)_2, b\o c \rangle
 &=& \langle \xi_{\s}(a), cb \rangle =\s(a, cb)\\
 &\stackrel{(\ref{eq:5.5})}{=}&\s(a_1, b)\s(a_2, c)= \langle \xi_{\s}(a_1), b \rangle  \langle \xi_{\s}(a_2), c \rangle \\
 &=& \langle \xi_{\s}(a_1)\o \xi_{\s}(a_2), b\o c \rangle ,
 \end{eqnarray*}
 finishing the proof.    \end{proof}

\begin{cor}\label{cor:5.13}
 Let $(C, \s)$ be a coquasitriangular infinitesimal Hopf algebra with bijective antipode $S$. Then
 \begin{eqnarray}
 && \s(id\o S)=\s(S^{-1}\o id);  \label{eq:5.8}\\
 && \s(S\o id)=\s(id\o S^{-1});  \label{eq:5.9}\\
 && \s(S\o S)=\s=\s(S^{-1}\o S^{-1}).  \label{eq:5.10}
 \end{eqnarray}
\end{cor}

\begin{proof}
 Since $\zeta_{\s}: C\lr C'$ and $\xi_{\s}: C\lr {'C}$ are infinitesimal Hopf algebra maps, then $\zeta_{\s}S=S^{-1*}\zeta_{\s}$ and $\xi_{\s}S=S^{-1*}\xi_{\s}$. For all $a, b\in C$,
 \begin{eqnarray*}
 \s(a, S(b))
 &=& \langle \zeta_{\s}(S(b)), a \rangle = \langle S^{-1*}(\zeta_{\s}(b)), a \rangle \\
 &=& \langle \zeta_{\s}(b), S^{-1}(a) \rangle =\s(S^{-1}(a), b),
 \end{eqnarray*}
 \begin{eqnarray*}
 \s(S(a), b)
 &=& \langle \xi_{\s}(S(a)), b \rangle = \langle S^{-1*}(\xi_{\s}(a)), b \rangle \\
 &=& \langle \xi_{\s}(a), S^{-1}(b) \rangle =\s(a, S^{-1}(b)),
 \end{eqnarray*}
 \begin{eqnarray*}
 &&\s(S\o S)=\s(id\o S)(S\o id)=\s(S^{-1}\o id)(S\o id)=\s
 \end{eqnarray*}
 and
 \begin{eqnarray*}
 &&\s(S^{-1}\o S^{-1})=\s(S^{-1}\o id)(id\o S^{-1})=\s(id\o S)(id\o S^{-1})=\s,
 \end{eqnarray*}
 finishing the proof.    \end{proof}

\begin{defi}\label{de:5.14}
 {\rm A {\bf double coalgebra} is a pair $(C, D)$ of coassociative coalgebras together with a left $D$-comodule on $C$: $C\lr D\o C$, $c\mapsto c_{-1}\o c_{0}$ and a right $C$-comodule on $D$: $D\lr D\o C$, $d\mapsto d_{(0)}\o d_{(1)}$ such that
 \begin{equation}\label{eq:5.11}
 c_{-1}\o c_{01}\o c_{02}=c_{1-1}\o c_{10}\o c_2+c_{-1(1)}\o c_{-1(0)}\o c_{0},
 \end{equation}
 \begin{equation}\label{eq:5.12}
 d_{(0)1}\o d_{(0)2}\o d_{(1)}=d_{1}\o d_{2(0)}\o d_{2(1)}+d_{(0)}\o d_{(1)-1}\o d_{(1)0},
 \end{equation}
 where $c\in C$ and $d\in D$.}
\end{defi}

\begin{thm}\label{thm:5.15}
 Given a double coalgebra $(C, D)$. Define the comultiplication $\D: C\o D\lr C\o D\o C\o D$ by
 \begin{equation}\label{eq:5.13}
 \D(c\o d)=c_{1}\o c_{2-1}\o c_{20}\o d+c\o d_{1(0)}\o d_{1(1)}\o d_{2}.
 \end{equation}
 Then $(C\o D, \D)$ is a coassociative coalgebra.
\end{thm}

\begin{proof} For all $c\in C$ and $d\in D$, we have
 \begin{eqnarray*}
 (\D\o id)\D(c\o d)
 &=&c_{11}\o c_{12-1}\o c_{120}\o c_{2-1}\o c_{20}\o d\\
 &&+c_{1}\o c_{2-11(0)}\o c_{2-11(1)}\o c_{2-1(2)}\o c_{20}\o d\\
 &&+c_{1}\o c_{2-1}\o c_{20}\o d_{1(0)}\o d_{1(1)}\o d_{2}\\
 &&+c\o d_{1(0)1(0)}\o d_{1(0)1(1)}\o d_{1(0)2}\o d_{1(1)}\o d_{2}\\
 &=&c_{1}\o c_{2-1}\o c_{20}\o c_{3-1}\o c_{30}\o d\\
 &&+c_{1}\o c_{2-1(0)}\o c_{2-1(1)}\o c_{20-1}\o c_{200}\o d\\
 &&+c_{1}\o c_{2-1}\o c_{20}\o d_{1(0)}\o d_{1(1)}\o d_{2}\\
 &&+c\o d_{1(0)1(0)}\o d_{1(0)1(1)}\o d_{1(0)2}\o d_{1(1)}\o d_{2}\\
 &\stackrel{(\ref{eq:5.12})}{=}&c_{1}\o c_{2-1}\o c_{20}\o c_{3-1}\o c_{30}\o d\\
 &&+c_{1}\o c_{2-1(0)}\o c_{2-1(1)}\o c_{20-1}\o c_{200}\o d\\
 &&+c_{1}\o c_{2-1}\o c_{20}\o d_{1(0)}\o d_{1(1)}\o d_{2}\\
 &&+c\o d_{11(0)}\o d_{11(1)}\o d_{12(0)}\o d_{12(1)}\o d_{2}\\
 &&+c\o d_{1(0)(0)}\o d_{1(0)(1)}\o d_{1(1)-1}\o d_{1(1)0}\o d_{2}\\
 &=&c_{1}\o c_{2-1}\o c_{20}\o c_{3-1}\o c_{30}\o d\\
 &&+c_{1}\o c_{2-1(0)}\o c_{2-1(1)}\o c_{20-1}\o c_{200}\o d\\
 &&+c_{1}\o c_{2-1}\o c_{20}\o d_{1(0)}\o d_{1(1)}\o d_{2}\\
 &&+c\o d_{1(0)}\o d_{1(1)}\o d_{2(0)}\o d_{2(1)}\o d_{3}\\
 &&+c\o d_{1(0)}\o d_{1(0)1}\o d_{1(1)2-1}\o d_{1(1)20}\o d_{2}\\
 &=&c_{1}\o c_{21-1}\o c_{210}\o c_{22-1}\o c_{220}\o d\\
 &&+c_{1}\o c_{2-1(0)}\o c_{2-1(1)}\o c_{20-1}\o c_{200}\o d\\
 &&+c_{1}\o c_{2-1}\o c_{20}\o d_{1(0)}\o d_{1(1)}\o d_{2}\\
 &&+c\o d_{1(0)}\o d_{1(1)}\o d_{2(0)}\o d_{2(1)}\o d_{3}\\
 &&+c\o d_{1(0)}\o d_{1(0)1}\o d_{1(1)2-1}\o d_{1(1)20}\o d_{2}\\
 &\stackrel{(\ref{eq:5.11})}{=}&c_{1}\o c_{2-1}\o c_{201}\o c_{202-1}\o c_{2020}\o d\\
 &&+c_{1}\o c_{2-1(0)}\o c_{2-1(1)}\o c_{20-1}\o c_{200}\o d\\
 &&+c_{1}\o c_{2-1}\o c_{20}\o d_{1(0)}\o d_{1(1)}\o d_{2}\\
 &&+c\o d_{1(0)}\o d_{1(1)}\o d_{2(0)}\o d_{2(1)}\o d_{3}\\
 &&+c\o d_{1(0)}\o d_{1(0)1}\o d_{1(1)2-1}\o d_{1(1)20}\o d_{2}\\
 &=&(id\o \D)\D(c\o d),
 \end{eqnarray*}
 finishing the proof.    \end{proof}

\begin{lem}\label{lem:5.16}
 Consider the maps:
 \begin{equation}\label{eq:5.14}
 \r_l: C\lr {'C}\o C, ~~c\mapsto e^{i}\o e_{i}c,
 \end{equation}
 \begin{equation}\label{eq:5.15}
 \r_r: {'C}\lr {'C}\o C, ~~c\mapsto -fe^{i}\o e_{i},
 \end{equation}
 where $c\in C, f\in {'C}$, $\{e_i\}$ is the basis of $C$ and $\{e^i\}$ is the dual basis.
 Then $(C, {'C})$ is a double coalgebra.
\end{lem}

\begin{proof} For all $c\in C$ and $f, g, h\in {'C}$,  we have
 \begin{eqnarray*}
 (id\o \r_{l})\r_{l}(c)
 &=&e^{i}\o e^{j}\o e_{j}e_{i}c\\
 &=&{e^{i}}_2\o {e^{i}}_1\o e_{i}c=(\D\o id)\r_{l}(c),
 \end{eqnarray*}
 so $(C, \r_l)$ is a left ${'C}$-comodule. Also
 \begin{eqnarray*}
 (\r_{r}\o id)\r_{r}(f)
 &=&fe^{i}e^{j}\o e_{j}\o e_{i}\\
 &=&-fe^{i}\o {e_{i1}}\o {e_{i2}}=(id\o \D)\r_{r}(f),
 \end{eqnarray*}
 so $({'C}, \r_r)$ is a right $C$-comodule.

 Applying $id\o g\o h$ to $fe^{i}e^{j}\o e_{j}\o e_{i}$,
 \begin{eqnarray*}
 fe^{i}e^{j} \langle g, e_{j} \rangle  \langle h, e_{i} \rangle
 &=&fhg\\
 &=&fe^{i} \langle -\D^{*op}(h\o g), e_{i} \rangle \\
 &=&-fe^{i} \langle \D^{*op}(h\o g), e_{i} \rangle \\
 &=&-fe^{i} \langle g, e_{i1} \rangle  \langle h, e_{i2} \rangle )
 \end{eqnarray*}

 Since
 \begin{eqnarray*}
 c_{-1}\o c_{01}\o c_{02}
 &=&e^{i}\o (e_{i}c)_1\o (e_{i}c)_2\\
 &=&e^{i}\o e_{i}c_1\o c_2+e^{i}\o e_{i1}\o e_{i2}c
 \end{eqnarray*}
 and
 \begin{eqnarray*}
 &&c_{1-1}\o c_{10}\o c_2+c_{-1(1)}\o c_{-1(0)}\o c_{0}\\
 &&~~~~~~~=e^{i}\o e_{i}c_1\o c_2-e^{i}e^{j}\o e_{j}\o e_{i}c,
 \end{eqnarray*}
 while for all $c'\in C$,
 \begin{eqnarray*}
  \langle e^{i}, c' \rangle e_{i1}\o e_{i2}c
 &=&c'_{1}\o c'_{2}c\\
 &=& \langle e^{j}, c'_{1} \rangle  \langle e^{i}, c'_{2} \rangle e_{j}\o e_{i}c\\
 &=& \langle -e^{i}e^{j}, c' \rangle e_{j}\o e_{i}c,
 \end{eqnarray*}
 thus Eq.(\ref{eq:5.11}) holds.

 Since
 \begin{eqnarray*}
 d_{(0)1}\o d_{(0)2}\o d_{(1)}
 &=&-(de^{i})_1\o (de^{i})_2\o e_{i}\\
 &=&-d{e^{i}}_1\o {e^{i}}_2\o e_{i}-d_1\o d_2e^{i}\o e_{i}
 \end{eqnarray*}
 and
 \begin{eqnarray*}
 &&d_{1}\o d_{2(0)}\o d_{2(1)}+d_{(0)}\o d_{(1)-1}\o d_{(1)0}\\
 &&~~~~~~~=-d_1\o d_2e^{i}\o e_{i}-de^{i}\o e^j\o e_je_{i}
 \end{eqnarray*}
 while for all $c', d'\in C$,
 \begin{eqnarray*}
  \langle d{e^{i}}_1, c' \rangle  \langle {e^{i}}_2, d' \rangle e_{i}
 &=&- \langle d, {c'}_{2} \rangle  \langle {e^{i}}_1, {c'}_{1} \rangle  \langle {e^{i}}_2, d' \rangle e_{i}\\
 &=&- \langle d, {c'}_{2} \rangle  \langle e^{i}, d'{c'}_{1} \rangle e_{i}\\
 &=&- \langle d, {c'}_{2} \rangle d'{c'}_{1}\\
 &=&- \langle d, {c'}_{2} \rangle  \langle e^i, {c'}_{1} \rangle d'e_i\\
 &=& \langle de^i, c' \rangle d'e_i\\
 &=& \langle de^i, c' \rangle  \langle e^{j}, d' \rangle e_je_i,
 \end{eqnarray*}
 so Eq.(\ref{eq:5.12}) holds.
 Therefore the proof is finished.   \end{proof}

\begin{thm}\label{thm:5.17}
 Let $C$ be a finite dimensional infinitesimal bialgebra. Consider the linear space
 $$
 \widehat{D(C)}:=C\o {'C}.
 $$
 Let $\{e_i\}$ be the basis of $C$ and $\{e^i\}$ the dual basis of ${'C}$. Define
 \begin{eqnarray}
 \D(c\o f)
 &=&c_{1}\o c_{2-1}\o c_{20}\o f+c\o f_{1(0)}\o f_{1(1)}\o f_{2}  \nonumber\\
 &=&c_1\o e^i\o e_ic_2\o f-c\o f_1e^i\o e_i\o f_2.   \label{eq:5.16}
 \end{eqnarray}
 Then $\widehat{D(C)}$ is a coalgebra.

 Furthermore, define $\s\in (\widehat{D(C)}\o \widehat{D(C)})^*$ by
 \begin{equation}\label{eq:5.17}
 \s(c\o f, d\o g)= \langle g, c \rangle  \langle f, d \rangle .
 \end{equation}
 Then $\s$ is a solution of coassociative Yang-Baxter equation. i.e., $(\widehat{D(C)}, \s)$ is a coquasitriangular infinitesimal bialgebra. Moreover
 \begin{equation}\label{eq:5.18}
 \mu_{\s}(c\o f, d\o g)=-c\o f_1 \langle f_2, d \rangle g-c \langle f, d_1 \rangle d_2\o g,
 \end{equation}
 where $c, d\in C$ and $f, g\in {'C}$.
\end{thm}

\begin{proof}  By Lemma \ref{lem:5.16}, we get the first asseration. It is straightforward that Eq.(\ref{eq:5.18}) is a solution of CAYB in  $\widehat{D(C)}$. \end{proof}

\begin{defi}\label{de:5.18}
 {\rm Let $(C, \d_1, \d_2, \mu)$ be a covariant bialgebra. A right $C$-module and a right $C$-comodule $M$ is said to be a {\bf right covariant $C$-module} if  the action is a right covariant coderivation with respect to $\d_1$. Symmetrically, a left $C$-module and a left $C$-comodule $N$ is said to be a {\bf left covariant $C$-module} if  the action is a left covariant coderivation with respect to $\d_2$. A {\bf morphism of covariant modules} is a map that is both $C$-linear and $C$-colinear. The category of right (left) covariant $C$-modules is denoted by $\mathcal{M}^C_C (_C ^CM)$.}
\end{defi}

\begin{ex}\label{ex:5.19}
 Let $(C, \d_{\s}, \d_{\tau}, \mu)$ be a coquasitriangular covariant bialgebra associated to a CYBP $(\s, \tau)$. Then every right $C$-comodule $M$ is a right covariant $C$-module with the action
 \begin{equation}\label{eq:5.19}
 \tr_M: M\o C\lr M, m\o c\mapsto m_{(0)}\s(m_{(1)}, c).
 \end{equation}
 Similarly, every left $C$-comodule $N$ is a left covariant $C$-module with the action
 \begin{equation}\label{eq:5.20}
 \tl_M: C\o N\lr N, c\o n\mapsto -\tau(c, m_{(-1)})m_{(0)}.
 \end{equation}
\end{ex}

\begin{proof} For all $m\in M$ and $c\in C$,
 \begin{eqnarray*}
 m\tr_M c_1\o c_2+m_{(0)}\o m_{(1)}\di_1 c
 &\stackrel{}{=}&m_{(0)}\o m_{(1)1}\s(m_{(1)2}, c) \\
 &=&m_{(0)(0)}\o m_{(0)(1)}\s(m_{(1)}\\
 &=&\r(m\tr_M c).
 \end{eqnarray*}
 So the action is a right covariant coderivation with respect to $\d_{\s}$. In what follows, we will show that $(M, \tr_M)$ is a right $C$-module. For all $c, d\in C$ and $m\in M$, we have
 \begin{eqnarray*}
 (\tr_M\ci (id\o \mu))(m\o c\o d)
 &\stackrel{}{=}&m\tr_M (cd)\\
 &\stackrel{}{=}&m_{(0)}\s(m_{(1)}, cd)\\
 &\stackrel{(\ref{eq:5.1})}{=}&m_{(0)}\s(m_{(1)1}, d)\s(m_{(1)2}, c)\\
 &=&m_{(0)(0)}\s(m_{(0)(1)}, d)\s(m_{(1)}, c)\\
 &=&(\tr_M\ci (\tr_M\o id))(m\o c\o d).
 \end{eqnarray*}
 Therefore,  the first statement is verified. Similarly, we  prove the second one. \end{proof}

Next, we describe the relation between Rota-Baxter cosystems and dendriform coalgebras.

\begin{defi}\label{de:6.1}
 \emph{(\cite{Fo})} {\rm A triple $(C, \D_{\rc}, \D_{\lc})$ consisting of a vector space $C$ and two $K$-linear operators $\D_{\rc}, \D_{\lc}: C \lr C\o C$ is called a {\bf dendriform coalgebra} if, for all $c \in C$,
 \begin{eqnarray}
 (\D_{\rc}\o id)\ci \D_{\rc}(c)&=&(id\o \D_{\rc}+id\o \D_{\lc})\ci \D_{\rc}(c), \label{eq:6.1}\\
 (\D_{\lc}\o id)\ci \D_{\rc}(c)&=&(id\o \D_{\rc})\ci \D_{\lc}(c), \label{eq:6.2}\\
 (\D_{\rc}\o id+\D_{\lc}\o id)\ci \D_{\lc}(c)&=&(id\o \D_{\lc})\ci \D_{\lc}(c). \label{eq:6.3}
 \end{eqnarray}}
\end{defi}

\begin{thm}\label{thm:6.2}
 Let $C$ be a coassociative coalgebra and $Q, T: C\lr C$ be $K$-linear maps. Define the $K$-linear maps $\D_{\rc}, \D_{\lc}: C\lr C\o C$ by
 \begin{equation}\label{eq:6.4}
 \D_{\rc}(c)=c_1\o T(c_2),~~~\D_{\lc}(c)=Q(c_1)\o c_2
 \end{equation}
 for all $c\in C$. Then:

 (1) If $(C, Q, T)$ is a Rota-Baxter cosystem, then $(C, \D_{\rc}, \D_{\lc})$ is a dendriform coalgebra.

 (2) If $C$ is a non-degenerate coalgebra and $(C, \D_{\rc}, \D_{\lc})$ is a dendriform coalgebra, then $(C, Q, T)$ is a Rota-Baxter cosystem.
\end{thm}

\begin{proof} (1) First, we prove that Eq.(\ref{eq:6.1}) holds as follows.
 \begin{eqnarray*}
 \hbox{RHS of Eq.(\ref{eq:6.1})}&\stackrel{}{=}&c_1\o T(c_2)_1\o T(T(c_2)_2)+c_1\o Q(T(c_2)_1)\o T(c_2)_2\\
 &\stackrel{(\ref{eq:3.2})}{=}&c_1\o T(c_{21})\o T(c_{22})\\
 &\stackrel{}{=}&c_{11}\o T(c_{12})\o T(c_{2})=\hbox{LHS of Eq.(\ref{eq:6.1})}.
 \end{eqnarray*}
 Eq.(\ref{eq:6.2}) can be checked easily by the coassociativity.
 \begin{eqnarray*}
 \hbox{LHS of Eq.(\ref{eq:6.3})}&\stackrel{}{=}&Q(Q(c_1)_1)\o Q(c_1)_2\o c_2+Q(c_1)_1\o T(Q(c_1)_2)\o c_2\\
 &\stackrel{(\ref{eq:3.1})}{=}&T(c_{11})\o T(c_{12})\o c_2\\
 &\stackrel{}{=}&T(c_{1})\o T(c_{21})\o c_{22}=\hbox{RHS of Eq.(\ref{eq:6.3})}.
 \end{eqnarray*}
 Thus $(C, \D_{\rc}, \D_{\lc})$ is a dendriform coalgebra.

 (2) It is straightforward by the proof of (1) and the non-degeneracy of $C$. \end{proof}

\begin{rmk}\label{rmk:5.22}
 In \cite{Fo}, the notions of dendriform bialgebra, codendriform bialgebra and bidendriform bialgebra were introduced by combining dendriform algebra and dendriform coalgebra structures. As mentioned above, Rota-Baxter (co)algebra can produce dendriform (co)algebra. Then an interesting question is whether there is a kind of Rota-Baxter ``bi"-algebras that lead to structures of dendriform bialgebra, codendriform bialgebra or bidendriform bialgebra.
\end{rmk}

\section{Examples of Rota-Baxter bialgebras and bisystems} \label{sec:ExamplesRBbialgbisyst}
In this section we provide further examples of Rota-Baxter bialgebras and bisystems.
\subsection{2-dimensional Examples}
\begin{ex}[2-dimensional bialgebra]
We consider the 2-dimensional bialgebra $K \times K$,
where  the multiplication is defined with respect to a basis $\{u_1,u_2\}$ by  $u_1\cdot u_i=u_i\cdot u_1=u_i$ for $i=1,2$ and $u_2\cdot u_2 =u_2$; and the comultiplication is defined by
$ \Delta(u_1)=u_1 \otimes u_1,\  \Delta(u_2)=u_2 \otimes u_2$. We set that $u_1$ is the unit and the counit is defined by
$ \varepsilon (u_1)=\varepsilon (u_2)=1$.
\begin{itemize}
\item The $(R,R)$-Rota-Baxter structures of weight $\lambda$ are given by
\begin{enumerate}
\item $R(u_1)=-\lambda u_1, \  R(u_2)=0, $
\item $R(u_1)=0, \  R(u_2)=-\lambda u_2, $
\item $R(u_1)=-\lambda u_1, \  R(u_2)=-\lambda u_2$.
\end{enumerate}

\item The $(R,Q)$-Rota-Baxter structures of weight $\lambda$ and $\gamma$ respectively are given by  the pairs $(R,Q)$  taken from the following list of $R$'s

\begin{enumerate}
\item $R(u_1)=0, \  R(u_2)=-\lambda u_2, $
\item $R(u_1)=-\lambda u_2, \  R(u_2)=-\lambda u_2, $
\item $R(u_1)=\lambda u_2, \  R(u_2)=0, $
\item $R(u_1)=0, \  R(u_2)=\lambda u_1-\lambda u_2, $
\item $R(u_1)=-2 \lambda u_1+\lambda u_2, \  R(u_2)=-\lambda u_1, $
\item $R(u_1)=-\lambda u_1, \  R(u_2)=0, $
\item $R(u_1)=-\lambda u_1, \  R(u_2)=-\lambda u_2, $
\item $R(u_1)=- \lambda u_1-\lambda u_2, \  R(u_2)=-\lambda u_2, $
\item $R(u_1)=- \lambda u_1+\lambda u_2, \  R(u_2)=0, $
\item $R(u_1)=-\lambda u_1, \  R(u_2)=-\lambda u_1, $
\item $R(u_1)= \lambda u_1-\lambda u_2, \  R(u_2)= \lambda u_1-\lambda u_2, $
\end{enumerate}
and the following list of $Q$'s
\begin{enumerate}
\item $Q(u_1)=0, \  Q(u_2)=-\gamma u_2,$
\item $Q(u_1)=-\gamma u_1, \  Q(u_2)=0,$
\item $Q(u_1)=-\gamma u_1, \  Q(u_2)=-\gamma u_2$.
\end{enumerate}

\item The bisystems are given by  the pairs $(R,S)$ and  pairs $(Q,T)$ taken from the following list for $(R,S)$

\begin{enumerate}
\item $R(u_1)=p_1u_2, \  R(u_2)=0, \ S(u_1)=p_1u_1+p_2u_2, \ S(u_2)=p_2 u_2, $
\item $R(u_1)=p_1u_2, \  R(u_2)=p_1u_2, \ S(u_1)=-p_2u_1+p_1u_2, \ S(u_2)=0, $
\item $R(u_1)=p_1u_2, \  R(u_2)=p_2u_2, \ S(u_1)=(p_1-p_2)u_1, \ S(u_2)=0, $
\item $R(u_1)=p_1u_1, \  R(u_2)=p_1u_1, \ S(u_1)=-p_2u_1+p_2u_2, \ S(u_2)=p_1u_1-p_1u_2, $
\item $R(u_1)=-p_1u_1+p_1u_2, \  R(u_2)=0, \ S(u_1)=p_2u_2, \ S(u_2)=p_2u_2, $
\item $R(u_1)=-p_1u_1+p_1u_2, \  R(u_2)=-p_1u_1+p_1u_2, \ S(u_1)=-(p_1+p_2)u_1+p_2u_2, \ S(u_2)=-p_1u_1, $
\item $R(u_1)=-p_1u_1+p_1u_2, \  R(u_2)=p_1u_1-p_1u_2, \ S(u_1)=p_1u_1, \ S(u_2)=p_1u_1, $
\item $R(u_1)=-(p_1+p_2)u_1+p_1u_2, \  R(u_2)=-p_2u_1, \ S(u_1)=-p_2u_1+p_2u_2, \ S(u_2)=-p_2u_1+p_2u_2, $
\item $R(u_1)=p_1u_1+p_2u_2, \  R(u_2)=p_2u_2, \ S(u_1)=p_1u_2, \ S(u_2)=0, $
\item $R(u_1)=p_1u_1, \  R(u_2)=0, \ S(u_1)=p_2u_2, \ S(u_2)=(-p_1+p_2)u_2, $
\item $R(u_1)=0, \  R(u_2)=-p_1u_2, \ S(u_1)=p_1u_1, \ S(u_2)=0, $
\item $R(u_1)=p_1u_2, \  R(u_2)=0, \ S(u_1)=p_1u_1-p_1u_2, \ S(u_2)=0, $
\item $R(u_1)=p_1u_2, \  R(u_2)=0, \ S(u_1)=p_1u_1-p_1u_2, \ S(u_2)=-p_1u_2, $
\item $R(u_1)=p_1u_2, \  R(u_2)=0, \ S(u_1)=p_1u_1, \ S(u_2)=0, $
\item $R(u_1)=p_1u_2, \  R(u_2)=0, \ S(u_1)=p_1u_1-p_1u_2, \ S(u_2)=0, $
\item $R(u_1)=p_1u_2, \  R(u_2)=0, \ S(u_1)=p_1u_1-p_1u_2, \ S(u_2)=-p_1u_2, $
\item $R(u_1)=p_1u_2, \  R(u_2)=p_1u_2, \ S(u_1)=p_1u_1-p_1u_2, \ S(u_2)=p_1u_1-p_1u_2, $
\item $R(u_1)=p_1u_2, \  R(u_2)=p_1u_2, \ S(u_1)=p_1u_1-p_1u_2, \ S(u_2)=p_1u_1-p_1u_2, $
\item $R(u_1)=p_1u_1-p_1u_2, \  R(u_2)=0, \ S(u_1)=p_1u_2, \ S(u_2)=0, $
\item $R(u_1)=p_1u_1-p_1u_2, \  R(u_2)=-p_1u_2, \ S(u_1)=p_1u_2, \ S(u_2)=0, $
\item $R(u_1)=p_1u_1, \  R(u_2)=0, \ S(u_1)=p_1u_2, \ S(u_2)=0, $
\item $R(u_1)=p_1u_1-p_1u_2, \  R(u_2)=0, \ S(u_1)=p_1u_2, \ S(u_2)=0, $
\end{enumerate}
where   $p_1,p_2$ are parameters, and the following list for $(Q,T)$

\begin{enumerate}
\item $Q(u_1)=0, \  Q(u_2)=q_1u_2, \ T(u_1)=q_2u_1, \ T(u_2)=0, $
\item $Q(u_1)=q_1u_1, \  Q(u_2)=0, \ T(u_1)=0, \ T(u_2)=q_2u_2, $
\item $Q(u_1)=q_1u_2, \  Q(u_2)=0, \ T(u_1)=0, \ T(u_2)=q_1u_2, $
\item $Q(u_1)=0, \  Q(u_2)=q_1u_1+q_1u_2, \ T(u_1)=q_1u_1+q_1u_2, \ T(u_2)=0, $
\item $Q(u_1)=0, \  Q(u_2)=q_1u_1, \ T(u_1)=q_1u_1, \ T(u_2)=0, $
\item $Q(u_1)=q_1u_1+q_1u_2, \  Q(u_2)=0, \ T(u_1)=0, \ T(u_2)=q_1u_1+q_1u_2, $

\end{enumerate}
where   $q_1,q_2$ are parameters.
\end{itemize}
\end{ex}

\subsection{3-dimensional and 4-dimensional Examples}
\begin{ex}[3-dimensional bialgebra]
We consider the 3-dimensional bialgebra $K \times K$,
where  the multiplication is defined with respect to a basis $\{u_1,u_2,u_3\}$ by  $u_1\cdot u_i=u_i\cdot u_1=u_i$ for $i=1,2,3$ and $u_2\cdot u_2 =u_2$,  $u_2\cdot u_3 =u_3\cdot u_2 =u_3$, $u_3\cdot u_3 =0$; and the comultiplication is defined by
$ \Delta(u_1)=u_1 \otimes u_1$, $  \Delta(u_2)=u_1 \otimes u_2+u_2 \otimes u_1-u_2 \otimes u_2$,  $ \Delta(u_3)=u_1 \otimes u_3+u_3 \otimes u_1-u_3 \otimes u_2$. We set that $u_1$ is the unit and the counit is defined by
$ \varepsilon (u_1)=1$, $ \varepsilon (u_2)=\varepsilon (u_3)=0$.
\begin{itemize}
\item The $(R,R)$-Rota-Baxter structures of weight $\lambda$ are given by

\begin{enumerate}
\item $R(u_1)=-\lambda u_1, \  R(u_2)=0,R(u_3)=0$,
\item $R(u_1)=0, \  R(u_2)=-\lambda u_2, \ R(u_3)=0$,
\item $R(u_1)=0, \  R(u_2)=0, \ R(u_3)=-\lambda u_3$,
\item $R(u_1)=-\lambda u_1, \  R(u_2)=-\lambda u_2,R(u_3)=0$,
\item $R(u_1)=0, \  R(u_2)=-\lambda u_2, \ R(u_3)=-\lambda u_3$,
\item $R(u_1)=-\lambda u_1, \  R(u_2)=0, \ R(u_3)=-\lambda u_3$,
\item $R(u_1)=-\lambda u_1, \  R(u_2)=-\lambda u_2,\  R(u_3)=-\lambda u_3 $,
\end{enumerate}

\item The $(R,Q)$-Rota-Baxter structures of weight $\lambda$ and $\gamma$ respectively are given by the  pairs $(R,Q)$  taken from the following list of $R$'s

\begin{enumerate}
\item $R(u_1)=0, \  R(u_2)=0,R(u_3)=-\lambda u_3$,
\item $R(u_1)=0, \  R(u_2)=-\lambda u_2,R(u_3)=0$,
\item $R(u_1)=0, \  R(u_2)=-\lambda u_2,R(u_3)=-\lambda u_3$,
\item $R(u_1)=-\lambda u_2, \  R(u_2)=-\lambda u_2,R(u_3)=0$,
\item $R(u_1)=-\lambda u_2, \  R(u_2)=-\lambda u_2,R(u_3)=-\lambda u_3$,
\item $R(u_1)=\lambda u_2, \  R(u_2)=0,R(u_3)=0$,
\item $R(u_1)=\lambda u_2, \  R(u_2)=0,R(u_3)=-\lambda u_3$,
\item $R(u_1)=0, \  R(u_2)=\lambda u_1-\lambda u_2,R(u_3)=0$,
\item $R(u_1)=0, \  R(u_2)=\lambda u_1-\lambda u_2,R(u_3)=-\lambda u_3$,
\item $R(u_1)=-2\lambda u_1+\lambda u_2, \  R(u_2)=-\lambda u_1,R(u_3)=0$,
\item $R(u_1)=-2\lambda u_1+\lambda u_2, \  R(u_2)=-\lambda u_1,R(u_3)=-\lambda u_3$,
\item $R(u_1)=-\lambda u_1, \  R(u_2)=0,R(u_3)=0$,
\item $R(u_1)=-\lambda u_1, \  R(u_2)=0,R(u_3)=-\lambda u_3$,
\item $R(u_1)=-\lambda u_1, \  R(u_2)=-\lambda u_2,R(u_3)=0$,
\item $R(u_1)=-\lambda u_1, \  R(u_2)=-\lambda u_2,R(u_3)=-\lambda u_3$,
\item $R(u_1)=-\lambda u_1-\lambda u_2, \  R(u_2)=-\lambda u_2,R(u_3)=0$,
\item $R(u_1)=-\lambda u_1-\lambda u_2, \  R(u_2)=-\lambda u_2,R(u_3)=-\lambda u_3$,
\item $R(u_1)=-\lambda u_1+\lambda u_2, \  R(u_2)=0,R(u_3)=0$,
\item $R(u_1)=-\lambda u_1+\lambda u_2, \  R(u_2)=0,R(u_3)=-\lambda u_3$,
\item $R(u_1)=-\lambda u_1, \  R(u_2)=-\lambda u_1,R(u_3)=0$,
\item $R(u_1)=-\lambda u_1, \  R(u_2)=-\lambda u_1,R(u_3)=-\lambda u_3$,
\item $R(u_1)=\lambda u_1-\lambda u_2, \  R(u_2)=\lambda u_1-\lambda u_2,R(u_3)=0$,
\item $R(u_1)=\lambda u_1-\lambda u_2, \  R(u_2)=\lambda u_1-\lambda u_2,R(u_3)=-\lambda u_3$,

\end{enumerate}
and the following list of $Q$'s
\begin{enumerate}
\item $Q(u_1)=-\gamma u_1, \  Q(u_2)=-\gamma u_2, \  Q(u_3)=-\gamma u_3$.
\item $Q(u_1)=-\gamma u_1, \  Q(u_2)=-\gamma u_2, \  Q(u_3)=0$.
\item $Q(u_1)=-\gamma u_1, \  Q(u_2)=0, \  Q(u_3)=-\gamma u_3$.
\item $Q(u_1)=0, \  Q(u_2)=-\gamma u_2, \  Q(u_3)=-\gamma u_3$.
\item $Q(u_1)=-\gamma u_1, \  Q(u_2)=0, \  Q(u_3)=0$.
\item $Q(u_1)=0, \  Q(u_2)=-\gamma u_2, \  Q(u_3)=0$.
\item $Q(u_1)=0, \  Q(u_2)=0, \  Q(u_3)=-\gamma u_3$.
\item $Q(u_1)=0, \  Q(u_2)=p u_2-\sqrt{-p(\gamma +p)}u_3, \  Q(u_3)=-\sqrt{-p(\gamma +p)}u_2-(\gamma+p) u_3$.
\item $Q(u_1)=0, \  Q(u_2)=p u_2+\sqrt{-p(\gamma +p)}u_3, \  Q(u_3)=\sqrt{-p(\gamma +p)}u_2-(\gamma+p) u_3$.
\item $Q(u_1)=-\gamma u_1, \  Q(u_2)=p u_2-\sqrt{-p(\gamma +p)}u_3, \  Q(u_3)=-\sqrt{-p(\gamma +p)}u_2-(\gamma+p) u_3$.
\item $Q(u_1)=-\gamma u_1, \  Q(u_2)=p u_2+\sqrt{-p(\gamma +p)}u_3, \  Q(u_3)=\sqrt{-p(\gamma +p)}u_2-(\gamma+p) u_3$,
\end{enumerate}
where $p$ is a parameter.\\

\item Examples of  bisystems are given by  pairs $(R,S)$ and  pairs $(Q,T)$ taken from the following list for $(R,S)$

\begin{enumerate}
\item $R(u_1)=p_1u_3, \  R(u_2)=p_1u_3,  \  R(u_3)=-p_2u_3,\\
S(u_1)=(p_2-p_3)u_1+p_3u_2, \ S(u_2)=p_2 u_2, \ S(u_3)=0. $
\item $R(u_1)=p_1u_3, \  R(u_2)=0,  \  R(u_3)=0,\\
S(u_1)=p_2u_2+(p_1+p_3)u_3, \ S(u_2)=p_2 u_2+p_3u_3, \ S(u_3)=p_2u_3. $
\item $R(u_1)=p_1u_2+p_2u_3, \  R(u_2)=p_2u_3,  \  R(u_3)=-p_3u_3,\\
S(u_1)=p_1u_1+p_3u_2, \ S(u_2)=p_3 u_2, \ S(u_3)=0. $
\item $R(u_1)=-p_1u_1+p_1u_2, \  R(u_2)=-p_2u_1+p_2u_2,  \  R(u_3)=0,\\
S(u_1)=-p_2u_1+p_3u_3, \ S(u_2)=-p_2 u_1+p_3u_3, \ S(u_3)=-p_2u_3. $
\item $R(u_1)=-p_1u_1+p_1u_2, \  R(u_2)=-p_1u_1+p_1u_2,  \  R(u_3)=p_2u_1-p_2u_2,\\
S(u_1)=-(p_1+p_3)u_1+p_3u_2, \ S(u_2)=-p_1 u_1, \ S(u_3)=p_2u_1. $
\item $R(u_1)=-(p_1+p_2)u_1+p_1u_2, \  R(u_2)=-p_2u_1,  \  R(u_3)=p_3u_1,\\
S(u_1)=-p_2u_1+p_2u_2, \ S(u_2)=-p_2 u_1+p_2 u_2, \ S(u_3)=p_3u_1-p_3u_2. $
\item $R(u_1)=p_1u_1-p_1u_2+p_2u_3, \  R(u_2)=p_1u_1-p_1u_2+p_2u_3,  \  R(u_3)=-p_1u_3,\\
S(u_1)=p_1u_2, \ S(u_2)=p_1 u_2, \ S(u_3)=0. $

\end{enumerate}
where   $p_1,p_2,p_3$ are parameters, an the following list for $(Q,T)$

\begin{enumerate}
\item $Q(u_1)=-q_1 u_1, \  Q(u_2)=0 , \  Q(u_3)=-q_1 u_3,\\
T(u_1)=0, \ T(u_2)=q_1u_2, \ T(u_3)=0.$
\item $Q(u_1)=-q_1 u_1+q_1u_2, \  Q(u_2)=-q_1 u_1+q_1u_2 , \  Q(u_3)=0,\\
T(u_1)=q_1u_2, \ T(u_2)=q_1u_1, \ T(u_3)=q_1u_3.$
\item $Q(u_1)=-q_1 u_1+q_1u_2, \  Q(u_2)=0 , \  Q(u_3)=-q_1 u_3,\\
T(u_1)=q_1 u_1+q_1u_2, \ T(u_2)=q_1 u_1+q_1u_2, \ T(u_3)=0.$
\item $Q(u_1)=-\frac{q_1^2+q_2^2}{q_1}u_1 , \  Q(u_2)=-\frac{q_1^2+q_2^2}{q_1}u_1 , \  Q(u_3)=\frac{q_1^2+q_2^2}{q_2}u_1,\\
T(u_1)=0, \ T(u_2)=q_1 u_2+q_2u_3, \ T(u_3)=q_2 u_2+\frac{q_2^2}{q_1}u_3.$
\item $Q(u_1)=-\frac{q_1^2+q_2^2}{q_1}u_1 +q_1u_2+q_2u_3, \  Q(u_2)=-\frac{2q_1^2+q_2^2}{q_1}u_1 , \  Q(u_3)=\frac{2q_1^2+q_2^2}{q_2}u_1,\\
T(u_1)=q_1u_1+q_1u_2+q_2u_3, \ T(u_2)=q_1u_1+q_1u_2+q_2u_3, \ T(u_3)=q_2u_1+q_2 u_2+\frac{q_2^2}{q_1}u_3.$
\item $Q(u_1)=-q_1 u_1+q_1u_2, \  Q(u_2)=-q_1 u_1+q_1u_2 , \  Q(u_3)=q_2 u_1-q_2u_2,\\
T(u_1)=q_1u_2, \ T(u_2)=q_1u_1, \ T(u_3)=0.$

\item $Q(u_1)=0, \  Q(u_2)=q_1 u_1-\frac{q_1q_3+q_2^2}{q_3}u_2+q_2u_3 , \  Q(u_3)=\frac{q_1q_3}{q_2}u_1+\frac{q_1q_3+q_2^2}{q_2}u_2-q_3u_3,\\
T(u_1)=\frac{q_3^2+q_2^2}{q_3}u_1, \ T(u_2)=q_3u_2+q_2u_3, \ T(u_3)=q_2u_2+\frac{q_2^2}{q_3}u_3.$

\end{enumerate}
where   $q_1,q_2,q_3$ are parameters.

\end{itemize}
\end{ex}

\begin{ex}
Now, we consider the unital Taft-Sweedler algebra generated by $g,x$ and the relations $(g^2=1,\ x^2=0, \ x g=-g x).$ The comultiplication is defined by $\Delta (g)=g\otimes g$ and   $\Delta (x)=x\otimes 1+g\otimes x$, the counit is given by $\varepsilon(g)=1,\ \varepsilon (x)=0.$

The bialgebra $T_2$ is 4-dimensional, it is defined with respect to a basis
$
\{u_1=1, \ u_2=g,\ u_3=x,\ u_4=gx\}
$
by the following table
which  describes multiplying the $i$th row elements  by the $j$th column elements.

\[
\begin{array}{|c|c|c|c|c|}
  \hline
   \ & u_1& u_2 & u_3 & u_4  \\ \hline
   u_1& u_1& u_2 & u_3 & u_4 \\ \hline
   u_2 &u_2 & u_1 & u_4 & u_3 \\ \hline
   u_3 &u_3 & -u_4 & 0 & 0  \\ \hline
   u_4 &u_4 & -u_3 & 0 & 0  \\ \hline
\end{array}
\]
and
\begin{eqnarray*}
&& \Delta(u_1)=u_1 \otimes u_1,\  \Delta(u_2)=u_2 \otimes u_2,\ \Delta(u_3)=u_3 \otimes u_1+u_2 \otimes u_3,\ \Delta(u_4)=u_4 \otimes u_2+u_1 \otimes u_4.
\end{eqnarray*}
\begin{equation*}
\varepsilon (u_1)=\varepsilon (u_2)=1,\quad \varepsilon (u_3)=\varepsilon (u_4)=0.
\end{equation*}
\begin{itemize}
\item The $(R,R)$-Rota-Baxter structures of weight $\lambda$ are given by

\begin{enumerate}
\item $R(u_1)=0, \  R(u_2)=0, R(u_3)=-\lambda u_3, \  R(u_4)=-\lambda u_4, \  $,
\item $R(u_1)=-\lambda u_1, \  R(u_2)=-\lambda u_2, R(u_3)=0, \  R(u_4)=0, \  $,
\item $R(u_1)=-\lambda u_1, \  R(u_2)=-\lambda u_2, R(u_3)=-\lambda u_3, \  R(u_4)=-\lambda u_4, \  $,
\end{enumerate}

\item The $(R,Q)$-Rota-Baxter structures of weight $\lambda$ and $\gamma$ respectively are given by  pairs $(R,Q)$  taken from the following list of $R$'s

\begin{enumerate}
\item $R(u_1)=0, \  R(u_2)=-p_1u_1+p_1u_2-\frac{(\lambda+p_1)(\lambda+p_1+p_2)}{p_3}u_3+\frac{(\lambda+p_1)(\lambda+p_2)}{p_3}u_4,
\\ R(u_3)=-p_3u_1+p_3u_2-(2\lambda+p_1+p_2)u_3+(\lambda+p_2)u_4,
\\Ê R(u_4)=-p_3u_1+p_3u_2-(\lambda+p_1+p_2)u_3+p_2u_4,$
\item $R(u_1)=-\lambda u_1, \  R(u_2)=(\lambda+p_1)u_1+p_1u_2-\frac{(\lambda+p_1)(\lambda+p_1+p_2)}{p_3}u_3+\frac{(\lambda+p_1)(\lambda+p_2)}{p_3}u_4,
\\ R(u_3)=p_3u_1+p_3u_2-(2\lambda+p_1+p_2)u_3+(\lambda+p_2)u_4,
\\Ê R(u_4)=p_3u_1+p_3u_2-(\lambda+p_1+p_2)u_3+p_2u_4,$
\item $R(u_1)=-\lambda u_1, \  R(u_2)=\lambda u_1+p_1u_3+\frac{p_1p_2}{\lambda+p_2}u_4,
\\ R(u_3)=-(\lambda+p_2)u_3-p_2u_4,
\\Ê R(u_4)=(\lambda+p_2)u_3+p_2u_4,$
\item $R(u_1)=-\lambda u_1, \  R(u_2)=\lambda u_1+\frac{\lambda(\lambda+p_1)}{p_2}u_3+\frac{\lambda(\lambda+p_1)}{p_2}u_4,
\\ R(u_3)=-p_2u_1-p_2u_2-(2\lambda+p_1)u_3-(\lambda+p_1)u_4,
\\Ê R(u_4)=p_2u_1+p_2u_2+(\lambda+p_1)u_3+p_1u_4,$

\item $R(u_1)=\frac{1}{2}\lambda u_1-\frac{1}{2}\lambda u_2 +p_1u_3+p_2u_4, \
R(u_2)=\frac{1}{2}\lambda u_1-\frac{1}{2}\lambda u_2-p_2u_3+p_1u_4,\
\\ R(u_3)=-\frac{1}{2}\lambda u_3-\frac{1}{2}\lambda u_4,
\\Ê R(u_4)=-\frac{1}{2}\lambda u_3-\frac{1}{2}\lambda u_4,$

\end{enumerate}
where $p_1,p_2,p_3$  are parameters,
and the following list of $Q$'s
\begin{enumerate}
\item $Q(u_1)=-\gamma u_1, \  Q(u_2)=-\gamma u_2, \  Q(u_3)=-\gamma u_3, \  Q(u_4)=-\gamma u_4,$
\item $Q(u_1)=-\gamma u_1, \  Q(u_2)=-\gamma u_2, \  Q(u_3)=-\gamma u_3, \  Q(u_4)=0,$
\item $Q(u_1)=-\gamma u_1, \  Q(u_2)=-\gamma u_2, \  Q(u_3)=0, \  Q(u_4)=0,$
\item $Q(u_1)=-\gamma u_1, \  Q(u_2)=0, \  Q(u_3)=0, \  Q(u_4)=0,$
\item $Q(u_1)=-\gamma u_1, \  Q(u_2)=-\gamma u_2, \  Q(u_3)=0, \  Q(u_4)=-\gamma u_4,$
\item $Q(u_1)=0, \  Q(u_2)=0, \  Q(u_3)=-\gamma u_3, \  Q(u_4)=-\gamma u_4$.
\item $Q(u_1)=0, \  Q(u_2)=-\gamma u_2, \  Q(u_3)=0, \  Q(u_4)=0,$
\end{enumerate}

\item Examples of bisystems are given by a pair $(R,S)$ and a pair $(Q,T)$ taken from the following list for $(R,S)$

\begin{enumerate}
\item $R(u_1)=-p_1u_1, \  R(u_2)=0, \ R(u_3)=0, \ R(u_4)=0, $
\\  $S(u_1)=p_2 u_4, \  S(u_2)=p_3u_4, \ S(u_3)=0, \ S(u_4)=p_1u_4, $
\item $R(u_1)=-(p_1+p_2)u_1, \  R(u_2)=-(p_1+p_2)u_2, \ R(u_3)=0, \ R(u_4)=0, $
\\  $S(u_1)=p_3 u_3-\frac{p_3p_2}{p_1}u_4, \  S(u_2)=-p_3 u_3+\frac{p_3p_2}{p_1}u_4,
\ S(u_3)=p_1u_3-p_2u_4, \ S(u_4)=-p_1u_3+p_2u_4, $
\item $R(u_1)=-(p_1+p_2)u_1, \  R(u_2)=-(p_1+p_2)u_2, \ R(u_3)=0, \ R(u_4)=0, $
\\  $S(u_1)=p_3 u_3+\frac{p_3p_2}{p_1}u_4, \  S(u_2)=p_3 u_3+\frac{p_3p_2}{p_1}u_4,
\ S(u_3)=p_1u_3+p_2u_4, \ S(u_4)=p_1u_3+p_2u_4, $
\item $R(u_1)=-p_1u_1, \  R(u_2)=-p_1u_2, \ R(u_3)=0, \ R(u_4)=0, $
\\  $S(u_1)=p_3 u_3+p_4u_4, \  S(u_2)=p_4 u_3+p_3u_4,
\ S(u_3)=p_1u_3, \ S(u_4)=p_1u_4, $
\item $R(u_1)=-p_1u_1, \  R(u_2)=-p_1u_2, \ R(u_3)=0, \ R(u_4)=0, $
\\  $S(u_1)=p_2 u_4, \  S(u_2)=-p_2u_4,
 \ S(u_3)=-p_1u_4, \ S(u_4)=p_1u_4, $
\item $R(u_1)=-p_1u_1, \  R(u_2)=-p_1u_2, \ R(u_3)=0, \ R(u_4)=0, $
\\  $S(u_1)=p_2 u_4, \  S(u_2)=p_2u_4,
\ S(u_3)=p_1u_4, \ S(u_4)=p_1u_4, $
\item $R(u_1)=-(p_1+p_2)u_1, \  R(u_2)=-(p_1+p_2)u_2, \ R(u_3)=0, \ R(u_4)=0, $
\\  $S(u_1)=0, \  S(u_2)=0,
\ S(u_3)=p_1u_3-p_2u_4, \ S(u_4)=-p_1u_3+p_2u_4, $
\item $R(u_1)=-2p_1u_1, \  R(u_2)=-2p_1u_2, \ R(u_3)=0, \ R(u_4)=0, $
\\  $S(u_1)=p_2u_3-p_2u_4, \  S(u_2)=-p_2u_3+p_2u_4,
\ S(u_3)=p_1u_3-p_1u_4, \ S(u_4)=-p_1u_3+p_2u_4, $
\item $R(u_1)=0, \  R(u_2)=0, \ R(u_3)=-p_1u_3, \ R(u_4)=-p_1u_4, $
\\  $S(u_1)=p_1u_1, \  S(u_2)=p_1u_2,
\ S(u_3)=0, \ S(u_4)=0, $
\item $R(u_1)=0, \  R(u_2)=0, \ R(u_3)=0, \ R(u_4)=-p_1u_4, $
\\  $S(u_1)=p_1u_1, \  S(u_2)=0,
\ S(u_3)=0, \ S(u_4)=0, $
\end{enumerate}
where   $p_1,p_2,p_3$ are parameters, and the following list for $(Q,T)$

\begin{enumerate}
\item $Q(u_1)=0, \  Q(u_2)=0, \ Q(u_3)=q_1u_2-\frac{q_1q_3}{q_2}, \ Q(u_4)=q_2 u_2-q_3 u_4, $
\\  $T(u_1)=q_4 u_1, \  T(u_2)=q_3 u_2, \ T(u_3)=0, \ T(u_4)=0, $
\item $Q(u_1)=0, \  Q(u_2)=q_1 u_2+q_2 u_3, \ Q(u_3)=q_3u_3, \ Q(u_4)=\frac{q_1q_4}{q_2} u_2+q_4 u_4, $
\\  $T(u_1)=-q_3 u_1, \  T(u_2)=0, \ T(u_3)=0, \ T(u_4)=0, $
\item $Q(u_1)=q_1 u_2+q_2 u_4, \  Q(u_2)=0, \ Q(u_3)=q_3u_2+\frac{q_3q_2}{q_1} u_4, \ Q(u_4)=q_4u_2+\frac{q_2q_4}{q_1} u_4, $
\\  $T(u_1)=0, \  T(u_2)=\frac{q_1^2-q_4q_2}{q_1} u_2, \ T(u_3)=0, \ T(u_4)=0, $
\item $Q(u_1)=0, \  Q(u_2)=q_1 u_2+q_2 u_3, \ Q(u_3)=-q_3u_3, \ Q(u_4)=0, $
\\  $T(u_1)=q_3u_1, \  T(u_2)=0, \ T(u_3)=0, \ T(u_4)=q_3u_4, $
\item $Q(u_1)=q_1 u_2+q_2 u_4, \  Q(u_2)=0, \ Q(u_3)=0, \ Q(u_4)=\frac{q_1(q_1-q_3)}{q_2}u_2+ (q_1-q_3)u_4, $
\\  $T(u_1)=0, \  T(u_2)=q_3 u_2, \ T(u_3)=q_3u_3, \ T(u_4)=0, $
\item $Q(u_1)=q_1 u_2+q_2 u_4, \  Q(u_2)=-q_3 u_2, \ Q(u_3)=0, \ Q(u_4)=\frac{q_1(q_1+q_3)}{q_2}u_2+ q_1u_4, $
\\  $T(u_1)=0, \  T(u_2)=0, \ T(u_3)=q_3u_3, \ T(u_4)=0, $

\item $Q(u_1)=q_1 u_1+q_1 u_2, \  Q(u_2)=0, \ Q(u_3)=q_2 u_2+q_1 u_3, \ Q(u_4)=- q_1u_4, $
\\  $T(u_1)=0, \  T(u_2)=q_1 u_1+q_1 u_2, \ T(u_3)=0, \ T(u_4)=0, $

\item $Q(u_1)=q_1 u_1+q_1 u_2, \  Q(u_2)=0, \ Q(u_3)=q_1 u_3-q_1 u_4, \ Q(u_4)=0, $
\\  $T(u_1)=0, \  T(u_2)=q_1 u_1+q_1 u_2, \ T(u_3)=0, \ T(u_4)=q_1 u_3-q_1 u_4, $

\item $Q(u_1)=q_1 u_1+q_1 u_2, \  Q(u_2)=0, \ Q(u_3)=q_1 u_3+q_1 u_4, \ Q(u_4)=0, $
\\  $T(u_1)=0, \  T(u_2)=q_1 u_1+q_1 u_2, \ T(u_3)=0, \ T(u_4)=-q_1 u_3-q_1 u_4, $

\item $Q(u_1)=q_1 u_1, \  Q(u_2)=-q_2u_2, \ Q(u_3)=0, \ Q(u_4)=-q_2 u_4, $
\\  $T(u_1)=0, \  T(u_2)=0, \ T(u_3)=q_2u_3, \ T(u_4)=0, $

\end{enumerate}
where   $q_1,q_2,q_3$ are parameters.
\end{itemize}
\end{ex}

 {\bf Acknowledgments}  This work was partially supported by China Postdoctoral Science Foundation (No.2017M611291), Foundation for Young Key Teacher by Henan Province (No. 2015GGJS-088) and Natural Science Foundation of Henan Province (No. 17A110007). T. Ma is grateful to the Erasmus Mundus project FUSION for supporting the postdoctoral fellowship visiting to M\"alardalen University, V\"asteras, Sweden and to Division of Applied Mathematics at the School of Education, Culture and Communication for cordial hospitality.

 \end{document}